\documentclass[11pt,a4paper,english]{article}
\usepackage{pgf,tikz,pgfplots}
\pgfplotsset{compat=1.15}
\usetikzlibrary{arrows}

\usepackage[left=2.45cm, top=2.45cm,bottom=2.45cm,right=2.45cm]{geometry}

\usepackage{mathtools,amssymb,amsthm,mathrsfs,color}
\usepackage[english]{babel}
\usepackage[english = british]{csquotes}
\usepackage{amsfonts}              
\numberwithin{equation}{section}
\numberwithin{figure}{section}

\usepackage{mathtools}
\usepackage{verbatim}

\usepackage{subfigure}

\newtheorem{theorem}{Theorem}[section]
\newtheorem{prop}[theorem]{Proposition}
\newtheorem{lemma}[theorem]{Lemma}

{\theoremstyle{definition}

}

\newcommand{\nc}{\newcommand}
\nc{\R}{{\mathbb R}}
\nc{\N}{{\mathbb N}}
\nc{\Z}{{\mathbb Z}}
\nc{\BP}{\mathbb{P}}
\nc{\BQ}{\mathbb{Q}}

\DeclareMathOperator{\Bin}{Bin}
\DeclareMathOperator{\Var}{\mathbb{V}}

\DeclareMathOperator{\BCov}{{Cov}}

\usepackage{times}

\usepackage{comment}

\begin{document}

\title{On a game of chance in
Marc Elsberg's thriller
\enquote{GREED}}
\author{Tamara Göll\footnote{Karlsruhe Institute of Technology (KIT), D-76131 Karlsruhe, Germany. E-Mail: tamara.goell@kit.edu}\,\,   and
Daniel Hug\footnote{Karlsruhe Institute of Technology (KIT), D-76131 Karlsruhe, Germany. E-Mail: daniel.hug@kit.edu}}

\date{\today}

\maketitle
\begin{abstract}
A (possibly illegal) game of chance, which is described in Chapter 14 of Marc Elsberg's thriller \enquote{GREED}, seems to offer an excellent chance of winning. However, as the gambling starts and evolves over several rounds, the actual experience of the
vast majority of the gamblers in a pub is strikingly different. We provide an analysis of this specific game and several of its variants by elementary tools of probability. Thus we also encounter an interesting threshold phenomenon, which is related to the transition from a profit zone to a loss area. Our arguments are motivated and illustrated by numerical calculations with Python. \smallskip


\end{abstract}

\section{Introduction}

In the summer of 2020, one of the authors (DH) spent a wonderful, albeit short, family holiday on one of the  Frisian North Sea  islands. At last one could enjoy nature with almost no worries and a brief period of the year when incidence numbers, exponential growth and R-factors had faded into the background. My wife had immersed herself in one of her books on the beach, while I relaxed and let my eyes wander over the expanse of the sea.  Finally, my wife, who teaches mathematics, turned to me and said, \enquote{You might be interested in this.} While reading Marc Elsberg's thriller  \enquote{GREED}, she had come across an interesting connection, which we initially discussed animatedly without pencil and paper. The following investigation finally arose from this first conversation. \medskip

Elsberg had his breakthrough as an author in 2012 with \enquote{BLACKOUT}. The book describes the scenario of a widespread collapse of power supply and its consequences. With \enquote{ZERO} and \enquote{HELIX} he confirmed his reputation as a master of the science thriller. In his eighth book with the title \enquote{GREED} \cite{Els}, Elsberg deals with economic concepts, findings and theories with a focus on the question whether comprehensive cooperation between economic partners and branches of industry could lead to greater prosperity for everybody. He relies on scientific work related to ergodicity economics of a group led by Ole Peters at the London Mathematical Laboratory, which was supported by the Nobel Prize winners Murray Gell-Mann and Ken Arrow \cite{PGM}. \medskip

 In his review of \enquote{GREED}, Edgar Fell \cite[translated from  German]{Fell}  comments on Elsberg's book:

 \begin{quote}
\enquote{In the course of this exciting story, more and more connections of an economic nature emerge. It is fascinating to see how the author succeeds in bringing complex economic and social issues closer to the reader.
Even the mathematical foundations of game theory are built into the plot in a  \enquote{playful} way. An illegal  game of chance in a bar, for example, offers the opportunity to take the first steps in this direction. It works like magic. Elsberg's captivating art of storytelling allows even readers who are completely untrained in mathematics to grasp his number games. The knowledge that is conveyed stimulates thought - about how modern forms of society actually work.}
\end{quote}

\medskip

 The aim of the following considerations is to present an elementary analysis of a game of chance (a bet) from  Chapter 14 of Elsberg's thriller \enquote{GREED},  mentioned in the review by Edgar Fell, and of variants of this bet. On the one hand, this offers the opportunity to apply some of the basic concepts of (elementary)  stochastics. In this respect it is fair to say that games of chance provided much of the inspiration behind the birth of probability theory; see the historical review in Ethier's  book on the Doctrine of Chances \cite{Ethier}.
 On the other hand, the investigation naturally leads to a threshold phenomenon (phase transition). Phenomena of this kind were originally observed in statistical physics, but also play an important role in the analysis of random graphs and random polytopes. For an elementary introduction to the topic of phase transitions in classical random graphs (Erdös--Renyi) we refer to \cite{PE}, the classical work \cite{ER} as well as the monographs and textbooks \cite{BB, Hofstad, Janson}. Threshold phenomena with random polytopes (e.g.~in high dimensions), random cones and connections to  optimization, data analysis and signal processing have been investigated in \cite{ALMT,Bonnet1, Bonnet2, DT09, DT, GKT, Gad, HS1, HS2, Piv}, for example.

\medskip

 In Section \ref{sec:Gier}, we provide a summary of those aspects from \enquote{GREED} which are relevant for the present discussion. We essentially focus on Chapter 14, in which Elsberg  stages the dynamics associated with the gambling scenario (the bet). In the following sections, we will analyse this particular gamble step by step. Here we start from the specific situation provided in the thriller. Motivated by our observations on the initial scenario described by Elsberg and a first quantitative analysis in Section \ref{sec3}, we generalise the underlying parameters of this game of chance (see Section \ref{sec:variations}). At first we examine the asymptotic behaviour of the expected net profit as the number $n$ of rounds tends to infinity (Elsberg's choice is $n=100$). Then we introduce general parameters $u$ and $d$ used to update the score after each round and find pairs $(d,u)$ (at least numerically if $n$ is finite) that define a fair game (Elsberg's choice corresponds to $d=0.6$, $u=1.5$ which results in an unfair game). The asymptotic analysis as $n$ tends to infinity leads to a surprising threshold phenomenon which marks the asymptotically sharp transition from the profit zone to the loss area. While the original version of the bet is based on successively tossing a fair coin, in  Section \ref{sec:Fakecoins} we also explore the effect a biased coin has on the outcome of the bet, which allows us to establish a similar, but more general asymptotic threshold property.  In addition, it can be seen from our investigation that a surprisingly small bias may already turn an unfavourable game into a fair one while keeping the other parameters fixed (see Figure \ref{fig erw gewinn p var}). By a thorough analysis of the asymptotic behaviour of the variance of the net profit as the number $n$ of rounds tends to infinity, we can even deduce the limit distribution of the net profit. It turns out that the limit distribution is deterministic, except for the case in which $(d,u)$ lies on the boundary between profit zone and loss area when we obtain a two-point distribution in the limit.  Finally, in Section  \ref{sec_simulationen} we illustrate some numerical simulations that motivate a small excursion to a generalised birthday problem. Throughout the paper, our arguments are motivated and illustrated by numerical calculations with Python (relevant source code can be selected from the arXiv version of the paper).

\section{Elsberg's game of chance \cite{Els} and some initial insights} \label{sec:Gier}

In Chapter 14 of \enquote{GREED} \cite{Els}, Elsberg describes the following scenario. \medskip

A group of people gathers in a bar in `Berlin Mitte'. A man (Fitzroy Peel, the croupier) offers the following bet to the rest of the group: A player starts with an initial score of $100$ points (units).  Afterwards, a coin is tossed one hundred times. In each round, the current score is increased by fifty percent, if the coin shows \enquote{heads}. Otherwise, the current score is reduced by forty percent. After one hundred rounds, there are two possibilities. If the final score exceeds $100$ points (units), the player wins and receives double his stake. The mentioned stake can be chosen by the player, but is not allowed to exceed one hundred euros. If the final score does not exceed one hundred points, the player loses. The payout in the case of a loss is not explicitly specified in \cite{Els} (although it seems natural to assume that Elsberg intended the player to lose his or her entire stake). \medskip

While one member of the group called \enquote{T-Shirt} wants to participate right away, another one (Jan) is sceptical at first. T-Shirt tries to convince Jan with the following explanation. In his opinion, one simply needs to take the mean of the possible outcomes. He therefore adds the possible percentages after one round (150 percent and 60 percent) and divides the sum by the number of possible outcomes (two). This yields a mean of 105 percent in each round. Jan and two other members of the group seem to be convinced by T-Shirts explanation and agree to join the game. \medskip

However, another member of the group speaks up and expresses his concerns about the neglect of probabilities in the previous explanation. In his opinion, one needs to consider that the coin shows \enquote{heads} or \enquote{tails}, each with  probability  $\frac{1}{2}$. Therefore, he multiplies the possible outcomes with the associated probabilities to arrive at an average gain of 105 percent in each round - the same result as before. \medskip

Finally, one more member of the group explains his point of view  on the suggested bet. He explains that the average increase of five percent in each round yields an average final score (expected outcome)  of 131.5 times the initial score. T-Shirt seems confident of his victory and the group starts the offered gamble.\medskip

In the following chapters, Elsberg describes the  reactions of the members of the group as the gambling evolves. Initially, a euphoric mood spreads in the room, as the majority of the players are  successful in the beginning. However, after a couple of rounds more and more people end up with low scores and finally only one player has a  score exceeding $100$ (units). Now the mood in the room shifts completely. There are violent accusations of cheating, leading even to a physical confrontation.\medskip

Subsequently (in Chapter 22), a first popular explanation of the preceding events is given. On the one hand, it is argued that the mean values calculated by some of the participants are not appropriate for analysing the game, as they do not describe the course of the game over time, and that they ignore the fact that the initial situation can be different in each round.  Here, the phenomenon of (lack of) ergodicity (the coincidence of temporal mean values and probabilistic averages) is used as an explanation. However, other points are perhaps more decisive for the analysis of the game.

\medskip

More helpful is the remark that with an initial score of 100 (units) and a loss of forty percent in the first round, the score is reduced to sixty.  Then, in order to reach again 100 (units) by winning in the second round, 66.67 percent instead of just 50 percent of the current score would be required. But even if a player wins in the first round and loses in the second round, the remaining score is only ninety. Ultimately, the basic error of the players is to calculate an expected value for the starting round and to conclude that $0.5\cdot 1.5 + 0.5\cdot 0.6 = 1.05$ is the factor by which the profit per round should increase.  Although the expected score at the end of the game varies exactly in this way (see below), the rules of the game do not state that the payout is double the stake  if the \textit{mean value} is greater than 100 (units), but if the \textit{actual outcome} of the game results in a score that is greater than $100$  (units). In addition, the prize in the event of a win is independent of the final score. It is only relevant whether the final score exceeds the initial score of $100$ (units).

\section{A quantitative analysis}\label{sec3}

We start by summarising and formalising the rules of the previously described game. Here we suggest a payout rule in the case of a loss which is more advantageous for the gamblers than a complete loss of the stake.

\begin{enumerate}
    \item The stake $a\le 100 $ (euros) is placed. The initial score is $100$ (units).
    \item The game consists of 100 rounds, each starting with the toss of a fair coin.
    \item If the coin shows \enquote{heads}, the current score is increased by $50\%$.\\
    Otherwise, the current score is reduced by $40\%$.
    \item This procedure will be continued until 100 rounds are completed.
    \item In the end, the player wins if the final score exceeds $100$ and receives a payout of twice the individual  stake. In this case the net profit equals the stake. Otherwise, the payout of the player is
    $$
    \text{stake}\cdot \frac{\text{final score}}{100},
    $$
hence the net profit equals
$$
    \text{stake}\cdot \frac{\text{final score}}{100}-\text{stake},
$$
which is negative (a loss), but not a complete loss of the stake. In other words, the higher the final score, the higher the percentage of the stake that the player keeps.
\end{enumerate}

\noindent
\textbf{Remark:} As said before, the payout in the event of a loss is not explicitly specified by Elsberg. If, in contrast to the situation described above, we agree that the player loses his or her entire stake in the event of a loss, then the outcome would be even more disadvantageous for the player. The analysis below  then  simplifies significantly as the consideration of the quantity $A(\ldots)$ in Section \ref{subsec betting on the edge} can be omitted.

\medskip

\noindent
{\bf Simulations:} Using Python (or a similar programming language), the  game can be simulated very easily. Various realisations are displayed in  Section \ref{sec_simulationen}.

\medskip
\noindent
{\bf A first analysis:} Let $a$ denote the stake. We assume that the coin tosses are done independently with a fair coin. If the coin shows \enquote{heads} $k$ times and \enquote{tails}  $(100-k)$ times during the $n=100$ rounds, for some $k\in \{0, \ldots, 100\}$, then the final score is given by
\begin{equation}
    b(k) = 100\cdot 1.5^k \cdot 0.6^{100-k}. \label{kapital ende}
\end{equation}
Note that the final score depends only on the numbers of \enquote{heads} and \enquote{tails} and not on the particular order in which they appear. \medskip

Using \eqref{kapital ende}, we can easily find the smallest integer $k$ necessary to win the game. A player wins the game, yielding a net profit of $a$ euros, if and only if the condition
\begin{equation}\label{eqneu2}
    b(k) = 100\cdot 1.5^k \cdot 0.6^{100-k} > 100
\end{equation}
is satisfied. Condition \eqref{eqneu2} can be  rewritten as
\begin{equation}
    \left(\frac{1.5}{0.6} \right)^k > \left(\frac{1}{0.6} \right)^{100}
\end{equation}
or
\begin{equation}\label{eq_cond_win}
    k > 100 \cdot \frac{\ln\left(\frac{5}{3} \right)}{\ln\left(\frac{5}{2} \right)},
\end{equation}
where $\ln$ denotes the natural logarithm. Clearly,  \eqref{eq_cond_win} does neither depend on the initial score $100$,  nor on the stake $a$, and \eqref{eqneu2} is equivalent to
\begin{equation}
    k > 100 \cdot \frac{\ln 5 - \ln 3}{\ln 5 - \ln 2} \approx 55.749.
\end{equation}
Consequently, a player receives a net payout of $a$ euros (keeping the initial stake) in the case where $k\geq 56$, whereas the net profit is
   $a\cdot 1.5^k\cdot 0.6^{100-k} -a< 0$ euros (which is a loss that depends on the particular number $k$) if $k\leq 55$.

\medskip

In the underlying scenario, we can easily compute the probability of a loss, which is given by
$$
\sum_{k=0}^{55}\binom{100}{k}\left(\frac{1}{2}\right)^{100}\approx 0.864,
$$
whereas the probability of a win is approximately given by $0.136$.\medskip

The previously determined probability of winning can also be seen in our numerical simulations (see Section \ref{sec_simulationen}). Figure \ref{fig simulation win} illustrates the outcomes of $100$ simulations of Elsberg's game of chance. Only $14$ out of the $100$ simulations were beneficial for the gamblers.

\medskip

Further, we can determine the expected net profit, i.e., the payout minus the stake, of a player. Hence, the expected net profit is given by
\begin{align*}
&\sum_{k=0}^{55}\binom{100}{k}\left(\frac{1}{2}\right)^{100}\cdot\left(a\cdot 1.5^k\cdot 0.6^{100-k} -a\right)+
\sum_{k=56}^{100}\binom{100}{k}\left(\frac{1}{2}\right)^{100}\cdot a\\
&=a\cdot \left(\frac{1}{2}\right)^{100}\left[\sum_{k=0}^{55}\binom{100}{k}\left(1.5^k\cdot 0.6^{100-k}-1\right)+
\sum_{k=56}^{100}\binom{100}{k}\right]\\
&=a\cdot\left[ -1+\sum_{k=0}^{55}\binom{100}{k}0.75^k\cdot 0.3^{100-k}+\left(\frac{1}{2}\right)^{99}\sum_{k=56}^{100}\binom{100}{k}\right]\\
&\approx -0.68\cdot a.
\end{align*}
This demonstrates that a player should ultimately expect a significant loss.

\section{Variations}\label{sec:variations}
The previous analysis shows that the participants of Elsberg's bet will experience a significant loss on average. Thus the question arises which rule or parameter underlying the gambling leads to this unfair situation and how the framework can be adjusted in order to make the gambling more (or even less) advantageous for the  participants. In the following, we will study the influence of different parameters of Elsberg's game of chance (to which we also refer as a \enquote{bet}, \enquote{gamble} or simply a \enquote{game}), or rather of the version of it employing our specific payout rule, starting with the number  $n$ of rounds.

\subsection{Number of rounds}

We will now analyse the influence of the number $n$ of rounds  on the expected net profit. \medskip

For $n=1$, the behaviour of the gamble coincides with the players' perception. The expected net profit is then given by $G(a,1) = 0.3\cdot a$. For $n=2$, the expected net profit is still positive, given by $G(a,2) = 0.4\cdot a$. However, the final score is only larger than $100$ if $k=2$. This event occurs with probability $\frac{1}{4}$ implying that the probability of a loss is given by $\frac{3}{4}$. After $n=6$ rounds, the expected net profit is negative for the first time. Although $G(a,7)$ is positive again, the expected net profit is strictly negative for $n\geq 8$. Figure \ref{fig erw gewinn n variabel} illustrates the expected net profit after $n$ rounds and shows an interesting behaviour. While the value $G(a,n)$ is not monotonic due to some jumps, we can clearly see a decreasing trend. This leads to the conjecture that the expected net profit converges to $-a$ as $n\to\infty$ (see Figure \ref{fig erw gewinn konvergent}). A formal proof of the asymptotic behaviour can be found below in Proposition \ref{prop asymptotischer gewinn n to inf}.

\medskip

\begin{figure}
\centering
\includegraphics[scale=1]{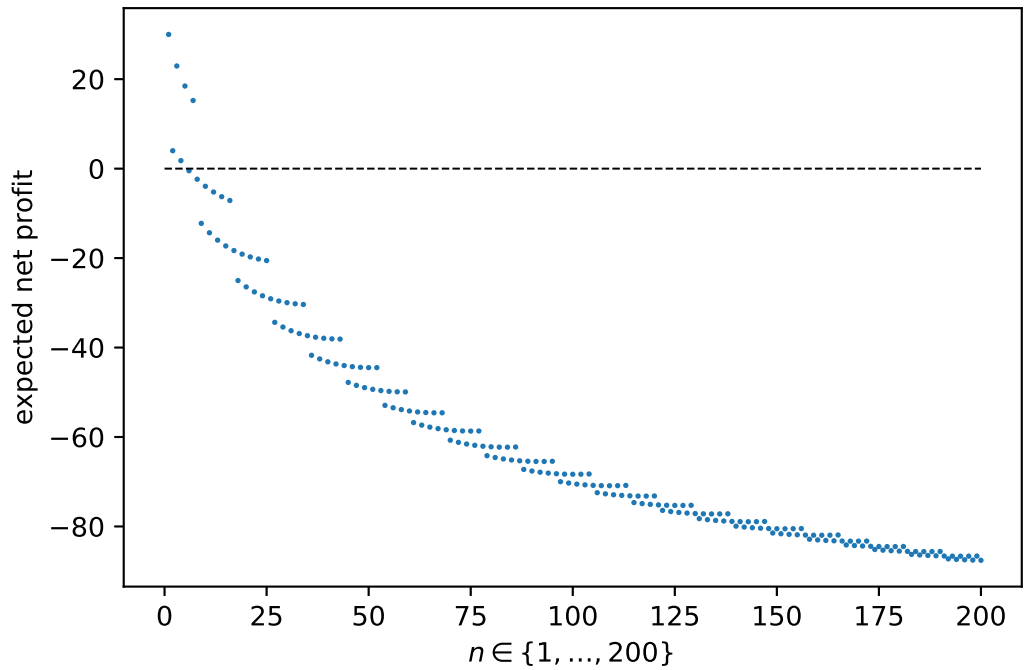}
\caption{Illustration of the expected net profit in Elsberg's bet after $n$ rounds ($n\in \{1,\ldots,200\}$) for an initial stake of $a=100$ euros.}
\label{fig erw gewinn n variabel}
\end{figure}

\begin{figure}
\centering
\includegraphics[scale=1]{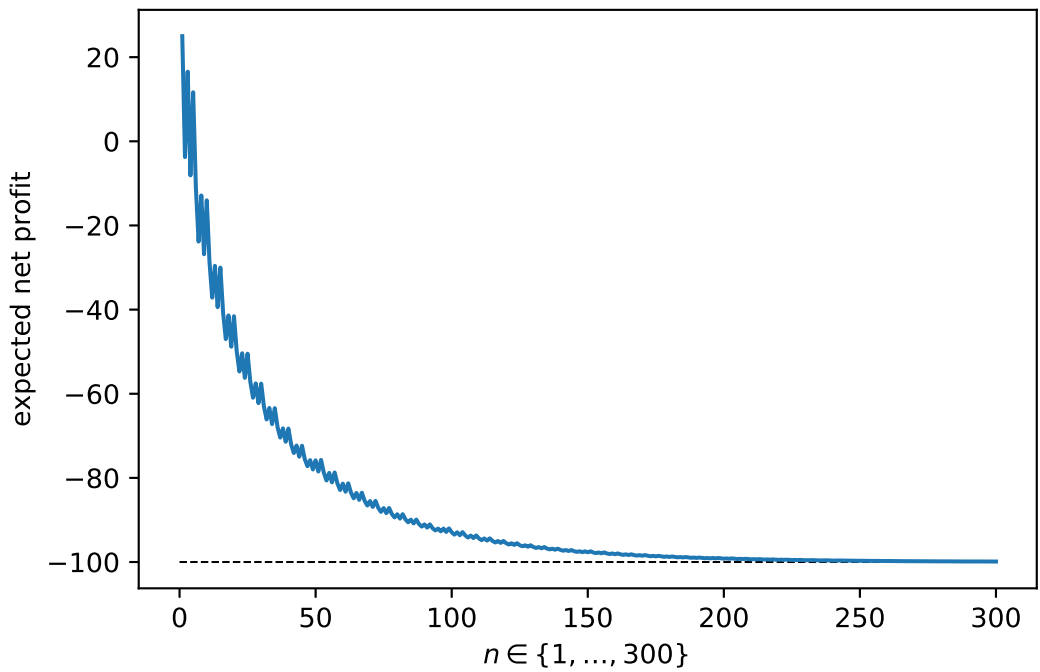}
\caption{Illustration of the expected net profit after $n$ rounds using an initial stake of $a=100$ (euros) in Elsberg's gamble.}
\label{fig erw gewinn konvergent}
\end{figure}

In general, the expected net profit after $n$ rounds with the initial stake $a$ is given by
\begin{equation}
G(a,n)\coloneqq a\cdot\left[-1+ \sum_{k=0}^{k_0(n)}\binom{n}{k}0.75^k\cdot 0.3^{n-k}+\left(\frac{1}{2}\right)^{n-1}\sum_{k=k_0(n)+1}^{n}\binom{n}{k}\right], \label{erw gewinn basic}
\end{equation}
where
$$
k_0(n)\coloneqq\left\lfloor n\cdot \frac{\ln 5-\ln 3}{\ln 5 - \ln 2}\right\rfloor
$$
represents the boundary between winning and losing events. In the definition of $k_0$, $\lfloor \cdot \rfloor$ denotes the floor function defined as $\lfloor x \rfloor = \max\{n\in \mathbb{Z}:\, n\leq x\}$, $x\in \mathbb{R}$. The logarithmic expression in the previous definition of $k_0$ is approximately given by
$$
\frac{\ln 5-\ln 3}{\ln 5 - \ln 2}\approx 0.5574929501.
$$

\bigskip

\noindent
We will now prove the previously mentioned conjecture regarding the asymptotic behaviour of the expected net profit using concentration inequalities. As usual, we denote a random variable $X$ with a binomial distribution with parameters $n\in \mathbb{N}$ and $p\in [0,1]$ by $X\sim \mathrm{Bin}(n,p).$

\begin{prop}\label{prop asymptotischer gewinn n to inf}
For any $a>0$,  $G(a,n)\rightarrow -a$ as $n\rightarrow \infty$ with an exponential rate of convergence.
\end{prop}

\begin{proof} In the following, we will prove that both sums in \eqref{erw gewinn basic} converge to zero as $n$ tends to infinity.

Let $X_n\sim \text{Bin}\left(n,\frac{5}{7}\right)$. Then we get
\begin{align}
\sum_{k=0}^{k_0(n)}\binom{n}{k}0.75^k\cdot 0.3^{n-k}
&=1.05^n\sum_{k=0}^{k_0(n)}\binom{n}{k}\left(\frac{5}{7}\right)^k\cdot \left(\frac{2}{7}\right)^{n-k}\allowdisplaybreaks\nonumber\\
&=1.05^n\cdot \mathbb{P}(X_n\le k_0(n))\allowdisplaybreaks\nonumber\\
&\le 1.05^n\cdot \mathbb{P}\left(X_n-\frac{5}{7}\, n\le -\left(\frac{5}{7}-0.5575\right)n\right)\allowdisplaybreaks\nonumber\\
&\le 1.05^n\cdot \exp\left(-2\cdot \left(\frac{5}{7}-0.5575\right)^2n\right)\allowdisplaybreaks\nonumber\\
&\le 1.05^n\cdot 1.050392097^{-n}\to 0\quad \text{ as $n\to\infty$},
\label{eqT1}
\end{align}
where we used Hoeffding's inequality \cite[Theorem 2.8]{BLM}, \cite{DP} in the second to last step. (Alternatively, Chernoff's inequality can be used, at the cost of an additional factor $2$.)

\medskip

Now let $Y_n\sim \text{Bin}\left(n,\frac{1}{2}\right)$. Using $k_0(n)\ge \lfloor 0.55n\rfloor\ge 0.55n-1$, we obtain the following upper bound on the second sum
\begin{align}
\left(\frac{1}{2}\right)^{n-1}\cdot \sum_{k=k_0(n)+1}^{n}\binom{n}{k} &= 2 \cdot \mathbb{P}(Y_n \geq k_0(n)+1)\nonumber\\
&\leq  2\cdot \mathbb{P}(Y_n\ge 0.55 n)\nonumber\\
&\le 2\cdot  \mathbb{P}(Y_n-0.5n\ge  0.05 n)\nonumber\\
&\le 2\exp\left(-2\cdot 0.05^2\cdot n\right)\to 0\quad \text{ as $n\to\infty$},\label{eqT2}
\end{align}
where we used Okamoto's inequality \cite{Oka}, \cite[Ex. 2.12]{BLM} in the last step.\medskip

An application of the upper bounds  \eqref{eqT1} and \eqref{eqT2} in \eqref{erw gewinn basic} completes the  proof.
\end{proof}

\medskip

\noindent
{\bf Remark:} If we are not interested in  the rate of convergence in Proposition \ref{prop asymptotischer gewinn n to inf}, the second part of the proof can be simplified by using the fact that $Y_n/n$ converges in probability to zero. However, it seems that for the first part of the argument some finer tools are required.

\medskip

An important feature of Elsberg's game of chance is the bounded payout in the event of a win. If the payout in the case of a loss would also apply in the winning scenarios, the expected net profit would be given by
\begin{equation}
\left(\frac{1}{2}\right)^{n}\sum_{k=0}^n \binom{n}{k}\left(1.5^k\cdot 0.6^{n-k}-1\right)\cdot a=a\cdot\left(1.05^n-1\right). \label{erw gewinn ohne deckelung basic}
\end{equation}

This is exactly the value the gamblers expected intuitively.

\subsection{Up and down}

In the following subsection, we will study the influence of the percentages used to modify the current score after each round. In Elsberg's game of chance, the score is increased by $50\%$ or decreased by $40\%$ if the coin shows \enquote{heads} or \enquote{tails}, respectively. We will substitute these percentages by some general percentages $a_u$ and $a_d$. The indices $u$ and $d$ stand for \enquote{up} and \enquote{down}.\medskip

Hence, the updated modification (increase or decrease) of the current score in each round is given as follows.

\begin{enumerate}
\item[3'.] If the coin shows \enquote{heads}, the current score is increased by $a_u\%$.\\
    Otherwise, the current score is reduced by $a_d\%$.
\end{enumerate}

In the game  introduced by Elsberg, the values $a_u$ and $a_d$ are apparently given by
$$ a_u = 50 \quad\text{and}\quad a_d = 40.$$

Since it will be more convenient to use fractions instead of percentages, we introduce the following factors
$$ u \coloneqq 1 + \frac{a_u}{100} \quad\text{and}\quad d \coloneqq 1 - \frac{a_d}{100}. $$

Using the previously defined factors $u$ and $d$,   we obtain more generally (cf. \eqref{kapital ende}) for the final score after $n$ rounds
\begin{equation}
\widetilde{b}(k) = 100 \cdot u^k \cdot d^{n-k}. \label{kapital ende neu}
\end{equation}

Here again, $k\in \{0,1,\ldots,n\}$ denotes the number of coin tosses showing \enquote{heads} among the $n$ independent  repetitions. Then,   the expected net profit under the updated modification rule 3' is given by
\begin{equation}
\widetilde{G}(a,n,u,d)= a \cdot \left(\frac{1}{2} \right)^n \left[\sum_{k=0}^{\widetilde{k}_0(n,u,d)} \binom{n}{k} \left(   {u}^k d^{n-k} - 1 \right) + \sum_{k = \widetilde{k}_0(n,u,d) + 1}^n \binom{n}{k} \right].\label{erw gewinn neu}
\end{equation}
As before, the quantity
\begin{equation}
\widetilde{k}_0(n,u,d) \coloneqq \left\lfloor -n \cdot \frac{\ln(d)}{\ln(u)-\ln(d)} \right\rfloor
\end{equation}
represents the boundary between winning and losing events. More precisely, a player wins if and only if $k> \widetilde{k}_0(n,u,d)$.\medskip

\noindent
{\bf General assumption:} In the following, we will always assume that $0 < d \leq 1 \leq u$ and $u \neq d$. This is not a loss of generality since other choices of $u$ and $d$ are not reasonable in the given situation.

\medskip

Intuitively, one would expect that an increase of $u$ or $d$ results in an advantage for the participants of the game of chance. The following Figures \ref{fig gewinn unbeschr d variabel} and \ref{fig gewinn unbeschr u variabel} support this conjecture. Both figures were generated using an initial stake of $a=100$ (euros) in a game of $n=100$ rounds. In Figure \ref{fig gewinn unbeschr d variabel}, we fixed the factor $u$ and illustrated the expected net profit in terms of $d$. In Figure \ref{fig gewinn unbeschr u variabel}, we treated the opposite situation where $d$ is fixed and the expected net profit is computed in terms of $u$. In both figures, the expected net profit in the situation described by Elsberg is marked by a red dot. \medskip

Both figures show that the expected net profit increases with the variable parameter $d$ and $u$, respectively. Moreover, both figures illustrate that the expected net profit approaches $-a$ or $a$ for small or large choices of $d$ and $u$, respectively. \medskip

Based on Figures \ref{fig gewinn unbeschr d variabel} and \ref{fig gewinn unbeschr u variabel}, one  expects the existence of choices for $u$ and $d$ that result in a fair game. In this context, we say that a game is \enquote{fair}, if the expected net profit equals zero. Numerically, it is possible to determine pairs $(d,u)$ which define a fair game. The blue line in Figure \ref{fig faire tupel} contains tuples $(d,u)$ resulting in a fair game. Tuples below the blue line are advantageous for the organiser of the game (the croupier) while tuples above the blue line result in a game which is advantageous for the participants. \medskip

For comparison, we also illustrated the function $u = d^{-1}$ (orange) suggesting the conjecture that  asymptotically as $n\to\infty$, the fair tuples $(d,u)$ are determined by the relation $u = d^{-1}$. We will prove this conjecture in Theorem \ref{Thm1} below. \medskip

Figure \ref{fig faire tupel} also contains a green dot marking the tuple $(0.6, 1.5)$ used in the game of chance introduced by Elsberg. This illustration shows once again that Elsberg's game of chance is unfair to the participants in the game.

\begin{figure}[th!]
\centering
\includegraphics[scale=1]{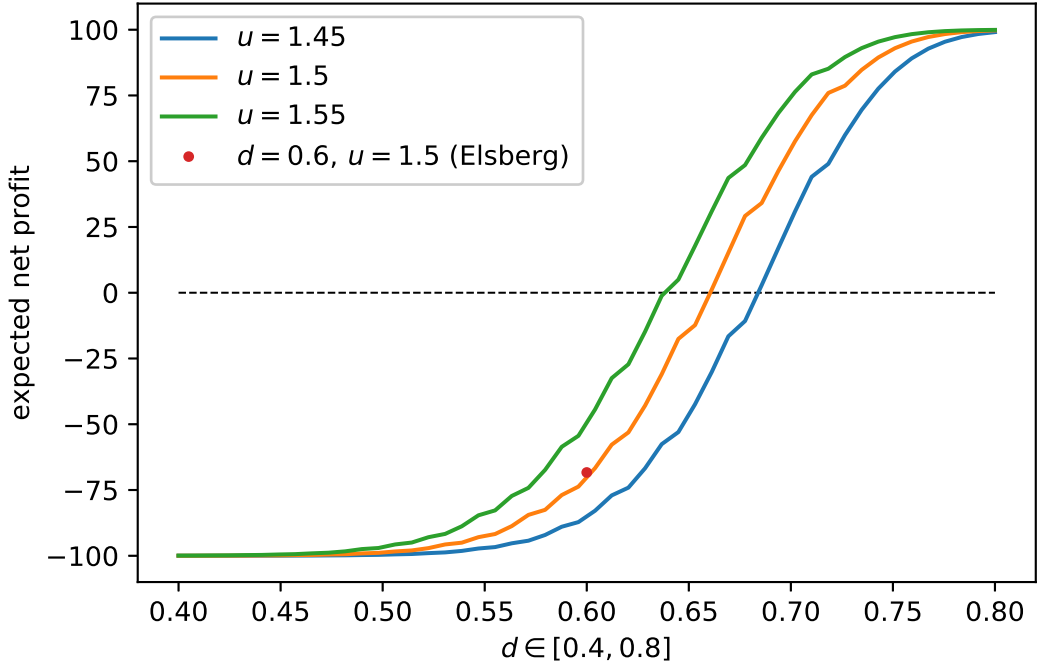}
\caption{Expected net profit after $n=100$ rounds with an initial stake of $a=100$ in terms of the \enquote{down factor} $d$ for various choices of the \enquote{up factor} $u$.}
\label{fig gewinn unbeschr d variabel}
\end{figure}

\begin{figure}[ht!]
\centering
\includegraphics[scale=1]{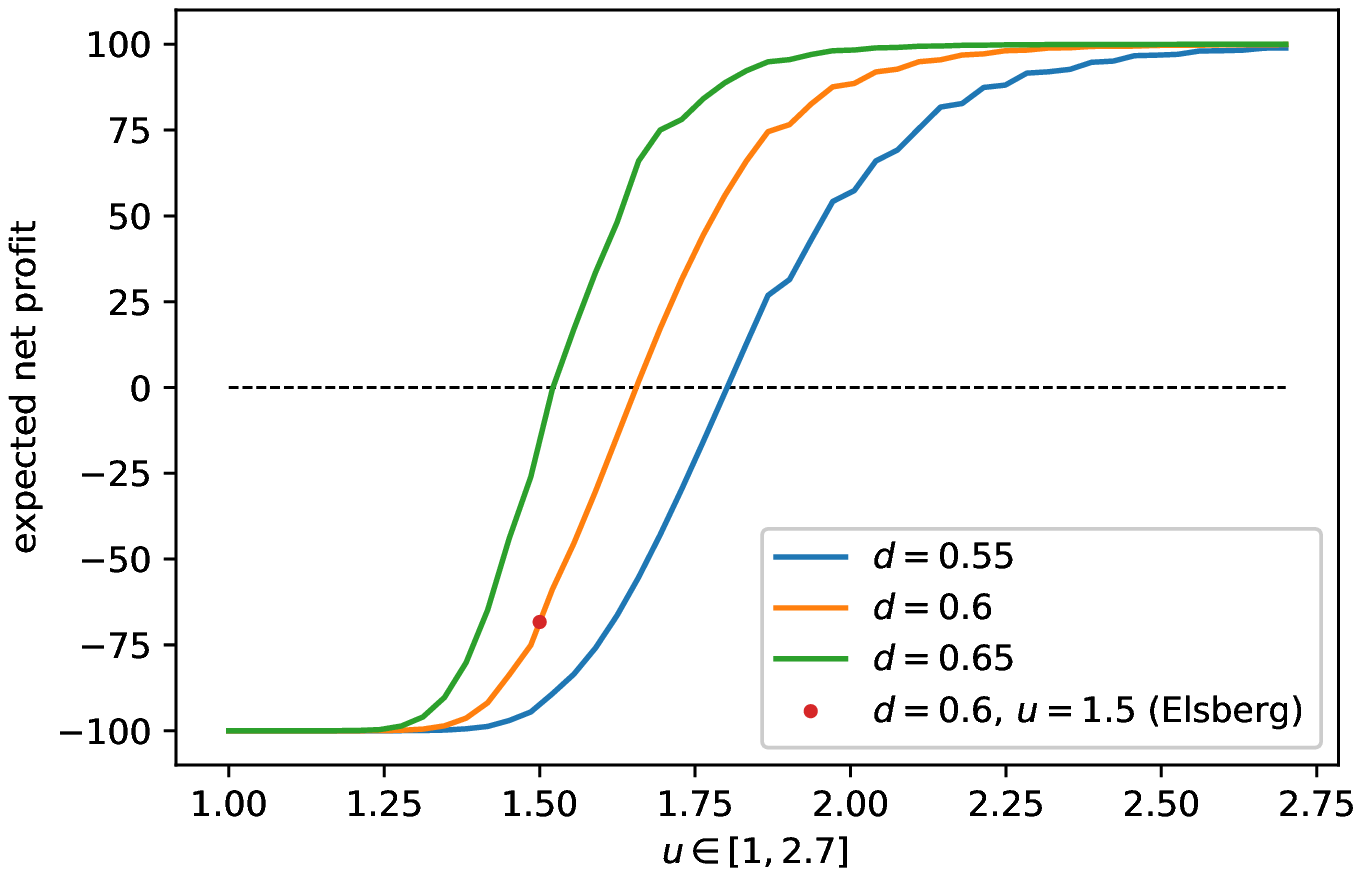}
\caption{Expected net profit after $n=100$ rounds with an initial stake of $a=100$ in terms of the \enquote{up factor} $u$ for various choices of the \enquote{down factor} $d$.}
\label{fig gewinn unbeschr u variabel}
\end{figure}

\begin{figure}[hb!]
\centering
\includegraphics[scale=1]{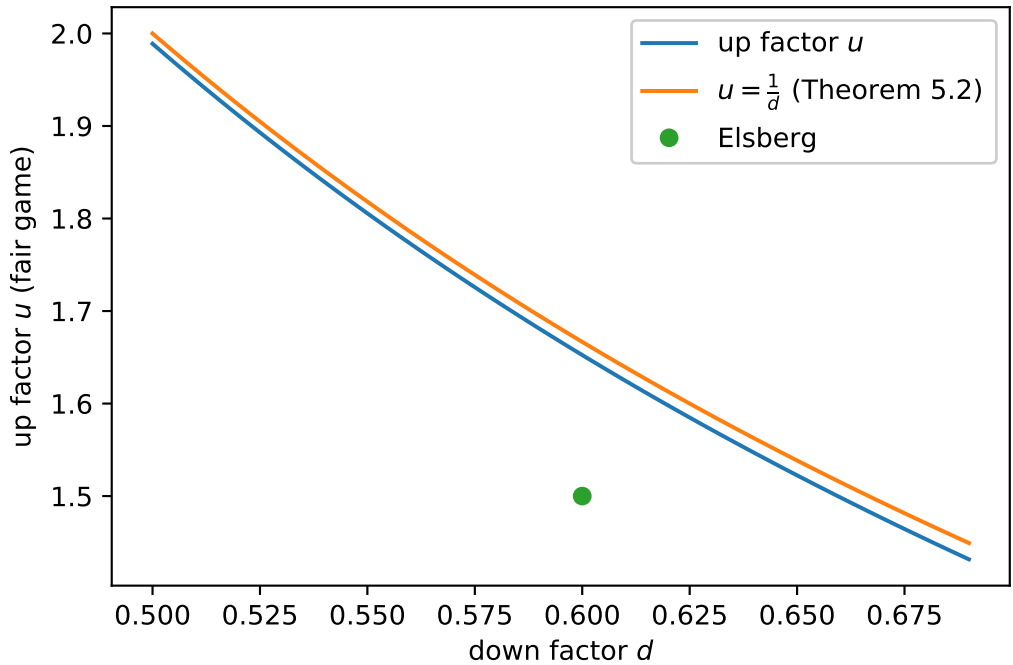}
\caption{Pairs $(d,u)$ resulting in a fair game for  $a=100$ and $n=100$ (blue curve).}
\label{fig faire tupel}
\end{figure}

\subsection{Gambling on the edge}\label{subsec betting on the edge}

As already mentioned in the previous subsection, the conjecture which is supported by the illustration of our numerical results in Figure \ref{fig faire tupel} can be proven exactly (asymptotically as $n\to\infty$). This leads to an interesting threshold phenomenon, where $\{(u,d)\in [1,\infty)\times (0,1]:ud=1,\ u\neq d\}$ describes the boundary between tuples leading to a game of chance that is advantageous or disadvantageous for the participants in the game. All tuples on the boundary lead to a fair game.\medskip

Since, according to the property $\widetilde{G}(a,n,u,d) = a \cdot \widetilde{G}(1,n,u,d)$, the expected net profit $\widetilde{G}(a,n,u,d)$ is proportional to the initial stake $a$, we define $G(n,u,d)\coloneqq \widetilde{G}(1,n,u,d)$ for notational simplicity. Then, it follows that $G(n,u,d)$ can be written as
$$
G(n,u,d)=-1+A(n,u,d)+B(n,u,d),
$$
where
$$
A(n,u,d)\coloneqq \left(\frac{1}{2} \right)^n\cdot \sum_{k=0}^{\widetilde{k}_0(n,u,d)}\binom{n}{k} {u}^k d^{n-k},  \qquad  B(n,u,d)\coloneqq \left(\frac{1}{2} \right)^{n-1}\cdot \sum_{k=\widetilde{k}_0(n,u,d)+1}^n\binom{n}{k}
$$
and
$$
\widetilde{k}_0(n,u,d)=\left\lfloor n \cdot \frac{\ln\left(\frac{1}{d}\right)}{\ln\left(\frac{u}{d}\right)} \right\rfloor,
$$
implying that $\widetilde{k}_0(n,u,d)\in\{0,\ldots,n\}$.

\begin{theorem}\label{Thm1}
Let $0<d\le 1\le u$ and $u\neq d$. Then
$$
\lim_{n\rightarrow \infty} G(n,u,d)= \begin{cases}
1,&\text{if } \, ud>1,\\
0,&\text{if } \, ud=1,\\
-1,&\text{if } \, ud<1.
\end{cases}
$$
\end{theorem}

The proof of Theorem \ref{Thm1} will be split into two auxiliary results concerning the asymptotic behaviour of $A(n,u,d)$ and $B(n,u,d)$ as $n\to\infty$.

\begin{lemma}\label{L1}
Let $0<d\le 1\le u$ and $u\neq d$. Then
$$
\lim_{n\rightarrow \infty} A(n,u,d)=0.
$$
\end{lemma}

 \begin{proof}
 \textbf{Case 1:} $u+d<2$. In this case,
$$
A(n,u,d)=\left(\frac{u+d}{2}\right)^n\cdot\sum_{k=0}^{\widetilde{k}_0(n,u,d)}\binom{n}{k}
\left(\frac{u}{u+d}\right)^k\left(\frac{d}{u+d}\right)^{n-k}\le \left(\frac{u+d}{2}\right)^n\to 0
$$
as $n\to\infty$.

\medskip

\textbf{Case 2:}  $u+d= 2$. Since $u\neq d$ is satisfied by assumption, it follows that $d<1$ must hold as well. We can therefore conclude that $2\cdot\left(1-\frac{d}{2}\right)^{1-\frac{d}{2}}\left(\frac{d}{2}\right)^{\frac{d}{2}}>1$ (see Lemma \ref{ungl1} for $x=\frac{d}{2}\in (0,\frac{1}{2}$)), and thus $(2-d)^{2-d}d^d>1$ or equivalently $u^u d^{2-u}>1$. This inequality is in turn  equivalent to
$$
\frac{u}{2}>\frac{\ln\left(\frac{1}{d}\right)}{\ln\left(\frac{u}{d}\right)},
$$
which implies that
\begin{equation}\label{b2}
\frac{1}{n}\widetilde{k}_0(n,u,d)\to \frac{\ln\left(\frac{1}{d}\right)}{\ln\left(\frac{u}{d}\right)}<\frac{u}{2}\quad \text{for }n\to\infty.
\end{equation}
Now let $Y_n\sim\Bin(n,\frac{u}{2})$. The law of large numbers yields
\begin{align*}
A(n,u,d)&=\sum_{k=0}^{\widetilde{k}_0(n,u,d)}\binom{n}{k}
\left(\frac{u}{2}\right)^k\left(\frac{d}{2}\right)^{n-k}
=\BP\left(Y_n\le \widetilde{k}_0(n,u,d)\right)\\
&=\BP\left(\frac{1}{n}Y_n\le \frac{\widetilde{k}_0(n,u,d)}{n}\right)\to 0 \quad\text{as }n\to\infty,
\end{align*}
due to \eqref{b2} and $\frac{1}{n}\mathbb{E} Y_n=\frac{u}{2}$. \medskip

If we use Hoeffding's inequality \cite[Theorem 2.8]{BLM} (or alternatively Chernoff's inequality including an additional factor $2$) and the fact that  $\frac{\widetilde{k}_0(n,u,d)}{n}-\frac{u}{2}<0 $ for sufficiently large $n$ due to \eqref{b2}, it is even possible to verify  an exponential rate of convergence. Using the previously mentioned tools, we get
$$
\BP\left(Y_n-\frac{nu}{2}\le \left(\frac{\widetilde{k}_0(n,u,d)}{n}-\frac{u}{2}\right)n\right)
\le \exp\left(-2\left(\frac{\widetilde{k}_0(n,u,d)}{n}-\frac{u}{2}\right)^2 n\right),
$$
where
$$
\frac{\widetilde{k}_0(n,u,d)}{n}-\frac{u}{2}\to
\frac{\ln\left(\frac{1}{d}\right)}{\ln\left(\frac{u}{d}\right)}
-\frac{u}{2}\eqqcolon \alpha(u,d)<0\quad\text{for }n\to\infty.
$$
Note that $\alpha(u,d)\to 0$ as $d,u \to 1$.

\medskip

\textbf{Case 3:} $u+d>2$ and $ud\neq 1$. It follows that $u>2-d\ge 1$, hence   $u^ud^d> (2-d)^{2-d}d^d\ge1$, and therefore we obtain
\begin{equation}\label{b4}
 \frac{\ln\left(\frac{1}{d}\right)}{\ln\left(\frac{u}{d}\right)}<\frac{u}{u+d}<1.
\end{equation}
For sufficiently large $n$ it also follows that
\begin{equation}\label{b3}
1-\frac{\widetilde{k}_0(n,u,d)}{n}>\frac{d}{u+d}\quad\text{or equivalently }\quad
\frac{\widetilde{k}_0(n,u,d)}{n}<\frac{u}{u+d}.
\end{equation}

Let $\xi_n\sim \Bin (n,\frac{u}{u+d})$ and $\zeta_n\coloneqq n-\xi_n\sim\Bin(n,\frac{d}{u+d})$. We will now use an inequality (see \cite[Ex. 2.11]{BLM} or \cite{DP}) that usually arises during the derivation of Chernoff's inequality. It states that if
  $S_n \sim \mathrm{Bin}(n,p)$ and $y\in (p,1)$, then
 \begin{equation}
     \mathbb{P}(S_n \geq ny) \leq \left(\frac{(1-p)^{1-y}p^y}{(1-y)^{1-y}y^y} \right)^n. \label{eq chernoff-hoeffding}
 \end{equation}

 (Remark: In the Wikipedia article \cite{wiki} this step is referred to as the \enquote{Chernoff-Hoeffding theorem}.) \medskip

Using \eqref{b3} and \eqref{eq chernoff-hoeffding} for $p=\frac{d}{u+d}$ and $y = 1 - n^{-1}\widetilde{k}_0(n,u,d)$, it follows that
\begin{align*}
A(n,u,d)&=\left(\frac{u+d}{2}\right)^n\BP\left(\xi_n\le
\widetilde{k}_0(n,u,d)\right)
=\left(\frac{u+d}{2}\right)^n\BP\left(\zeta_n\ge n-
\widetilde{k}_0(n,u,d)\right)\\
&=\left(\frac{u+d}{2}\right)^n\BP\left(\zeta_n\ge \left(1-
\frac{\widetilde{k}_0(n,u,d)}{n}\right)n\right)\\
&\le \left(\frac{u+d}{2}\right)^n
\left[\frac{\left(\frac{u}{u+d}\right)^{\frac{\widetilde{k}_0(n,u,d)}{n}}
\left(\frac{d}{u+d}\right)^{1-\frac{\widetilde{k}_0(n,u,d)}{n}}}
 {\left(\frac{\widetilde{k}_0(n,u,d)}{n}\right)^{\frac{\widetilde{k}_0(n,u,d)}{n}}
\left(1-\frac{\widetilde{k}_0(n,u,d)}{n}\right)^{1-\frac{\widetilde{k}_0(n,u,d)}{n}}}
\right]^n\\
&=\omega(n,u,d)^n,
\end{align*}
where
$$
\omega(n,u,d)\coloneqq \frac{u^{\frac{\widetilde{k}_0(n,u,d)}{n}}
d^{1-\frac{\widetilde{k}_0(n,u,d)}{n}}}
{2\left(\frac{\widetilde{k}_0(n,u,d)}{n}\right)^{\frac{\widetilde{k}_0(n,u,d)}{n}}
\left(1-\frac{\widetilde{k}_0(n,u,d)}{n}\right)^{1-\frac{\widetilde{k}_0(n,u,d)}{n}}}.
$$
Since we assumed that $ud \neq 1$, we get
$$
\frac{1}{n}\widetilde{k}_0(n,u,d)\to \frac{\ln\left(\frac{1}{d}\right)}{\ln\left(\frac{u}{d}\right)}
\in\left(0,1\right)\setminus \left\{\tfrac{1}{2}\right\},
$$
and therefore
$$
2\left(\frac{\widetilde{k}_0(n,u,d)}{n}\right)^{\frac{\widetilde{k}_0(n,u,d)}{n}}
\left(1-\frac{\widetilde{k}_0(n,u,d)}{n}\right)^{1-\frac{\widetilde{k}_0(n,u,d)}{n}}\to\rho(u,d)\in (1,2) \quad \text{as }n\to\infty.
$$
Furthermore,
\begin{equation}\label{eqhilf2}
u^{\frac{\widetilde{k}_0(n,u,d)}{n}}
d^{1-\frac{\widetilde{k}_0(n,u,d)}{n}}\to u^{\frac{\ln\left(\frac{1}{d}\right)}
{\ln\left(\frac{u}{d}\right)}}d^{\frac{\ln\left(u\right)}
{\ln\left(\frac{u}{d}\right)}}=1 \quad \text{for }n\to\infty.
\end{equation}
Finally, this implies that $\omega(n,u,d)\to  {\rho(u,d)}^{-1}\in (0,1)$, which yields $\omega(n,u,d)^n\to 0$ for $n\to\infty$, as requested.

\medskip

\textbf{Case 4:}  $u+d> 2$ and $ud=1$. If  $n$ is even,  there exists some $m \in \mathbb{N}$ such that $n=2m$. Hence,  $\widetilde{k}_0(n,u,d)=\left\lfloor \frac{n}{2}\right\rfloor=m$ and therefore $A(n,u,d)$ can be written as
$$
A(n,u,d)=\sum_{k=0}^{m}\left(\frac{1}{2}\right)^{2m}\binom{2m}{k}d^{2m-2k}.
$$
Since $k\le m$, it follows that $\binom{2m}{k}\le \binom{2m}{m}$. Now we can use Stirling's approximation to see that
\begin{equation}\label{stirling}
\binom{2m}{m}\sim\frac{2^{2m}}{\sqrt{\pi m}},
\end{equation}
where $\sim$ means that the two expressions are asymptotically equivalent. Hence, for $k\le m$ and sufficiently large $m\in\N$,
$$
\left(\frac{1}{2}\right)^{2m}\binom{2m}{k}\le \frac{1}{\sqrt{m}}.
$$
We conclude that
$$
A(n,u,d)\le \frac{1}{\sqrt{m}}\sum_{j=0}^{m} (d^2)^j\le (1-d^2)^{-1}\frac{1}{\sqrt{m}}\to 0 \quad \text{as } n\to \infty.
$$

If $n$ is odd,  there exists some $m\in\mathbb{N}$ such that $n=2m+1$. Then it follows that $\widetilde{k}_0(n,u,d)=m$ and $A(n,u,d)$ can be bounded by
\begin{align*}
    A(n,u,d)&=\sum_{k=0}^{m}\left(\frac{1}{2}\right)^{2m+1}\binom{2m+1}{k}d^{2m+1-2k}\\
    &=\frac{d}{2}\sum_{k=0}^{m}\left(\frac{1}{2}\right)^{2m}\frac{2m+1}{2m+1-k}\binom{2m}{k}d^{2m-2k}\\
    &\le \frac{d}{2}\frac{2m+1}{m+1}A(2m,u,d)\le A(2m,u,d)\to 0\quad \text{as } n\to \infty,
\end{align*}
which completes the argument.
 \end{proof}

  \noindent
 \textbf{Remarks.} $\bullet$ As soon as \eqref{b2} is available, the second case in the previous proof can be included into the third case. However, the application of Hoeffding's inequality is easier in the second case. Moreover, we could alternatively argue with the law of large numbers in the second case (though without obtaining the exponential rate of convergence then). Therefore we decided to treat these cases separately.

 \medskip

$\bullet$ At the critical boundary, characterised by the equation $ud = 1$, the expression $A(n,u,d)$ still converges to zero. However, the rate of convergence is no longer exponential, but of order $1/\sqrt{n}$. We only presented an upper bound on the order of convergence, but one could consider the summand for $k=m$ to deduce a lower bound as well.

\medskip

In order to establish the asymptotic behaviour of $G(n,u,d)$, which is stated in Theorem \ref{Thm1}, it remains to analyse the asymptotic behaviour of $B(n,u,d)$. The required result is provided by the following lemma.

\begin{lemma}\label{L2}
If $0<d\le 1\le u$ and $u\neq  d$, then
$$
\lim_{n\rightarrow \infty} B(n,u,d)=
\begin{cases}
0,& \text{if } \,  ud<1,\\
1,& \text{if } \,  ud=1,\\
2,& \text{if } \,  ud>1.
\end{cases}
$$
\end{lemma}

\begin{proof}
\textbf{Case 1:}  $ud<1$. Then
\begin{equation}\label{b1}
\frac{1}{n}\widetilde{k}_0(n,u,d)\to \frac{\ln\left(\frac{1}{d}\right)}{\ln\left(\frac{u}{d}\right)}>\frac{1}{2}\quad \text{as }n\to\infty,
\end{equation}
where the lower bound on the limit is equivalent to $ud<1$. If we introduce binomially distributed random variables $X_n\sim\Bin(n,\frac{1}{2})$, it follows that
$$
B(n,u,d)=2\,\BP\left(X_n\ge \widetilde{k}_0(n,u,d)+1\right)\le 2\,\BP\left(\frac{1}{n}X_n\ge \frac{1}{n}\widetilde{k}_0(n,u,d)\right)\to 0\quad\text{for }n\to\infty,
$$
by the law of large numbers. An application of Okamoto's inequality (see \cite{Oka}), in combination with the fact that $\widetilde{k}_0(n,u,d)-\frac{n}{2}>0$  for sufficiently large $n\in\N$ and \eqref{b1}, yields the even stronger statement
$$
B(n,u,d)\le 2\,\BP\left(X_n-\frac{n}{2}\ge \widetilde{k}_0(n,u,d)-\frac{n}{2}\right)
\le 2\exp\left(-2\left(\frac{\widetilde{k}_0(n,u,d)}{n}-\frac{1}{2}\right)^2n\right)\to 0
$$
as $n\to\infty$.

\medskip

\textbf{Case 2:}  $ud>1$. Then
$$
\frac{1}{n}\widetilde{k}_0(n,u,d)\to \frac{\ln\left(\frac{1}{d}\right)}{\ln\left(\frac{u}{d}\right)}
<\frac{1}{2} \text{ as } n\to \infty.
$$
Analogously to the first case, an application of the law of large numbers yields
$$
B(n,u,d)=2\, \BP(X_n\ge \widetilde{k}_0(n,u,d)+1)=
2\, \BP\left(\frac{1}{n}X_n\ge \frac{1}{n}(\widetilde{k}_0(n,u,d)+1)\right)\to 2\cdot 1=2
$$
for $n\to\infty$.

\medskip

Again, an application of Okamoto's inequality \cite{Oka} instead of the law of  large numbers, leads to a stronger, exponential estimate given by
\begin{align*}
B(n,u,d) &= 2 \, \mathbb{P}\left(X_n \ge \widetilde{k}_0(n,u,d)+1 \right) = 2 \left(1- \mathbb{P}\left(X_n \le \widetilde{k}_0(n,u,d) \right) \right)\\
 &\geq 2 \left(1- \exp \left(-2n\left(\frac{\widetilde{k}_0(n,u,d)}{n}-\frac{1}{2} \right)^2 \right) \right) \rightarrow 2\quad \text{as }n\to\infty.
\end{align*}
In combination with the upper bound $B(n,u,d)\leq 2$, it now follows that $\lim_{n\rightarrow \infty} B(n,u,d)=2$,  and the convergence is of exponential order.

\medskip

\textbf{Case 3:} $ud=1$. Then we have  $\widetilde{k}_0(n,u,d)=\left\lfloor \frac{n}{2}\right\rfloor$. \medskip

If $n$ is odd, there exists some $m\in \mathbb{N}$ such that $n=2m+1$. From the identity
$$
\sum_{k=m+1}^{2m+1}\binom{2m+1}{k}=2^{2m}
$$
we deduce that
$$
B(n,u,d) =2\,\BP\left(X_n\ge \left\lfloor \frac{n}{2}\right\rfloor+1\right)=2\cdot \frac{1}{2}=1.
$$
If $n$ is even, hence $n=2m$ for some $m\in \mathbb{N}$, we get
$$
\sum_{k=m+1}^{2m}\binom{2m}{k}=2^{2m-1}-\frac{1}{2}\binom{2m}{m} ,
$$
and therefore, by Stirling's approximation \eqref{stirling}, it follows that
$$
B(n,u,d) =2\,\BP\left(X_n\ge \left\lfloor \frac{n}{2}\right\rfloor+1\right)\to 2\cdot \frac{1}{2}=1 \quad \text{as }n\to\infty.
$$
In this case, the convergence is of the order $1/\sqrt{n}$.
\end{proof}

\subsection{Unbounded prize}

As we already mentioned in Section \ref{sec3}, the a priori  bound on the  prize in the event of a win contributes to the disadvantages of the participants in Elsberg's game of chance. We will now analyse a modified gambling rule, which  depends on the final score and does not involve an a priori bound. \medskip

We will now assume that the payout at the end of the game is always given by $a$ times the final score. In analogy to the representation of the expected net profit in \eqref{erw gewinn ohne deckelung basic}, we can derive a general representation of the expected net profit using the updated payout rule in the case of a win. We therefore arrive at the following (much simpler) expression for  the expected net profit, that is,
\begin{equation}
a\cdot \left(\frac{1}{2} \right)^n \sum_{k=0}^n \binom{n}{k}  \left(u^k \cdot d^{n-k} - 1 \right) = a \cdot \left(\left(\frac{u+d}{2} \right)^n - 1 \right).\label{erw gewinn ohne deckelung neu}
\end{equation}

Clearly, the expected net profit is zero if and only if $u+d = 2$. Thus, using the current payout rule in the event of a win, we can exactly characterise the tuples $(d,u)$ which lead to a fair game, by the condition $u+d = 2$.

\medskip

In the subsection below, we will analyse a modified version of Elsberg's game of chance based on a coin which is not necessarily fair. Before pursuing this topic,  we  describe the influence of a biased coin in the current scenario of a payout which does not involve an a priori bound on the prize in the event of a win. If $p\in [0,1]$ describes the probability of the event \enquote{heads} and $q=1-p$ the probability that the coin shows \enquote{tails}, the expected net profit in the underlying situation is then given by $a((pu+qd)^n-1)$. Hence, under these assumptions the game is fair if and only if $pu + qd = 1$.

\subsection{Fake coins}\label{sec:Fakecoins}

In this subsection, we analyse the influence of the probability $p$ that the coin shows \enquote{heads}. In \enquote{GREED}, the gambling is executed using a fair coin ($p=\frac{1}{2}$). Instead of a fair coin, we could use a bent (biased) coin or some completely different Bernoulli experiment (for example, by tossing a drawing pin or a dice and modifying the rule in step 3' appropriately). For the sake of simplicity, we will continue to use the toss of a coin. \medskip

Intuitively, increasing the probability $p$ that the coin shows \enquote{heads} should increase the  chances of winning for the participants of the gambling and hence increase their expected net profit. Figure \ref{fig erw gewinn p var} supports this conjecture. The figure illustrates the expected net profit after $n=100$ rounds using an initial stake of $a=100$ in terms of the probability $p$. The factors $u$ and $d$ are chosen according to the gamble described in \enquote{GREED}.\medskip

As in the analysis of the influence of the up and down factors $u$ and $d$, we want to choose the probability $p$ so that a fair game of chance is obtained. Figure \ref{fig erw gewinn p var} suggests that it is possible (at least numerically) to determine such a probability $p$. The net profit for $p=\frac{1}{2}$ and the choice of $p$ which results in a fair game   are both marked in Figure \ref{fig erw gewinn p var}.

\begin{figure}[ht!]
\centering
\includegraphics[scale=1]{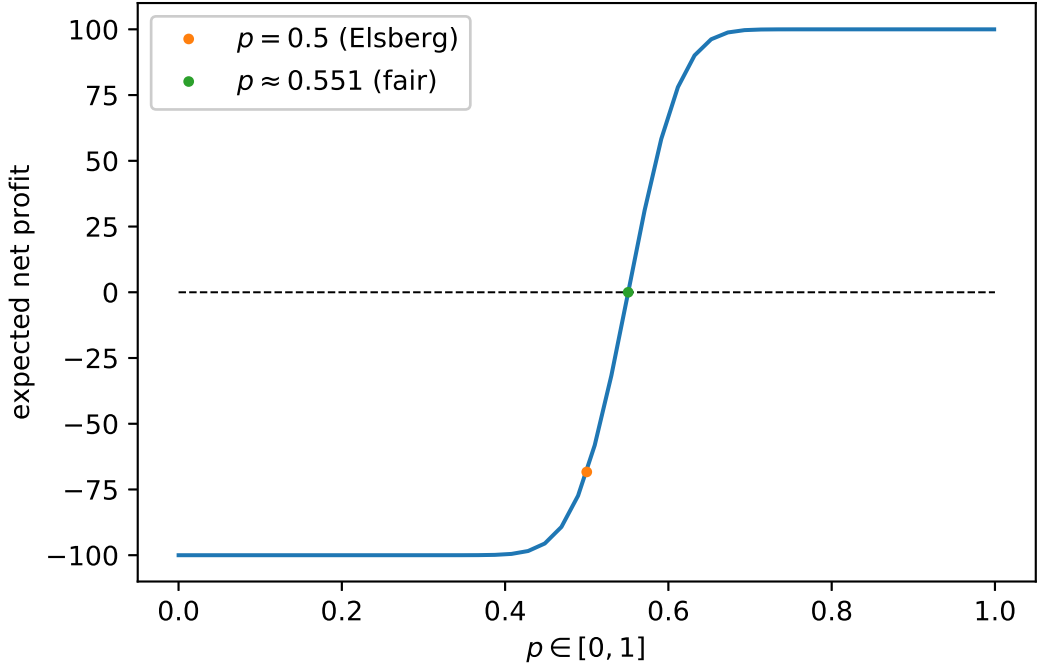}
\caption{Illustration of the expected net profit after $n=100$ coin tosses in terms of the probability $p$ that the coin shows \enquote{heads}, using the game parameters $a=100$, $u=1.5$ and $d=0.6$.}
\label{fig erw gewinn p var}
\end{figure}

In order to formalise the previously described modified version of Elsberg's game of chance, we consider a sequence of independent Bernoulli distributed random variables $X'_i \sim \Bin(1,p)$, where $\{X'_i = 1\}$ and $\{X'_i = 0\}$ represent the events that the $i$-th coin toss shows \enquote{heads} and \enquote{tails}, respectively. In order to simplify our calculations, we define the complementary probability $q=1-p$, which denotes the probability of the event \enquote{tails}. Then, the random variable $X_n^{(p)} \coloneqq X'_1 + \cdots + X'_n \sim \Bin(n,p)$ counts how often the event \enquote{heads} occurs among the $n$ coin tosses. \medskip

Since we are interested in the net profit at the end of the game, we introduce the random variable $T(n,u,d,p)$ given by
\begin{align*}
T(n,u,d,p):&=\sum_{k=0}^{\widetilde{k}_0(n,u,d)}\mathbf{1}\{X_n^{(p)}=k\}(u^kd^{n-k}-1)+
\mathbf{1}\{X_n^{(p)}\ge \widetilde{k}_0(n,u,d)+1\}\\
&=-1+ \sum_{k=0}^{\widetilde{k}_0(n,u,d)}\mathbf{1}\{X_n^{(p)}=k\} u^kd^{n-k} +2\cdot
\mathbf{1}\{X_n^{(p)}\ge \widetilde{k}_0(n,u,d)+1\},
\end{align*}
which represents the net profit after $n$ rounds using the initial stake $a=1$. At this point, we recall the explanation at the beginning of Subsection \ref{subsec betting on the edge} according to which the expected net profit is proportional to the stake $a$. Hence, it suffices to consider the case $a=1$.\medskip

By $\mathbf{1}\{X_n^{(p)}=k\}$ we denote the indicator function with respect to the event $\{X_n^{(p)}=k\}$. More precisely, the expression is given by
\begin{equation*}
    \mathbf{1}\{X_n^{(p)}=k\} = \begin{cases}
    1, & \text{if } X_n^{(p)}=k,\\
    0, & \text{if } X_n^{(p)}\neq k.
    \end{cases}
\end{equation*}

For the sake of notational simplicity, we will use the shorthand notation $T_n = T(n,u,d,p)$.\medskip

Since we defined $T_n$ as the net profit at the end of the game, the expected net profit is given by the expectation of the random variable $T_n$. Hence, we get the following representation of the expected net profit $G(n,u,d,p)\coloneqq \mathbb{E}[T(n,u,d,p)]$, that is,
\begin{align*}
G(n,u,d,p)&=-1+  \sum_{k=0}^{\widetilde{k}_0(n,u,d)}\binom{n}{k}(pu)^k(qd)^{n-k}+
2\cdot \sum_{k=\widetilde{k}_0(n,u,d)+1}^{n}\binom{n}{k}p^kq^{n-k}\\
&\eqqcolon -1+ A(n,u,d,p)+  B(n,u,d,p).
\end{align*}
Here it should be noted that $\widetilde{k}_0(n,u,d)$ is independent of $p$.

\medskip

Similarly to the results in Section \ref{subsec betting on the edge}, we determine the asymptotic behaviour of the expected net profit as $n \rightarrow \infty$. Again, for the proof we consider two auxiliary results concerning the asymptotic behaviour of the quantities $A(n,u,d,p)$ and  $B(n,u,d,p)$ as $n\to\infty$.

\medskip

We start by providing two analytic inequalities which will be useful in the proof of Lemma \ref{L1p}.

\begin{lemma}\label{ungl1}
If $d\in (0,1]$ and $x\in [0,d]$, then
$$
(1-x)^{1-x}(d-x)^x\ge 1-\frac{x}{d}.
$$
The inequality is strict if $d\in (0,1)$ and $x\in (0,d)$.

\end{lemma}

\begin{proof}
If $d=1$, the assertion of the lemma is apparently true, since the expressions on the left- and right-hand side are equal. Now let $d\in (0,1)$ be arbitrary, but fixed. We introduce the auxiliary function
 $$
 f(x)\coloneqq(1-x)\ln(1-x)+x\ln(d-x)-\ln
 \left(1-\frac{x}{d}\right),\quad x\in [0,d).
 $$
 Then, $f(0)=0$ and
 $$
 f'(x)=\frac{1-x}{d-x}-1-\ln\left(\frac{1-x}{d-x}\right)>0\quad \text{for } x\in [0,d),
 $$
 since  $h-1-\ln(h)>0$ for  $h>1$. Using the strict monotonicity of $f$, it follows that $f(x)> 0$ for  $x\in (0,d)$. Now the assertions of the lemma can be easily deduced.
\end{proof}

\begin{lemma}\label{ungl2}
If $x,p\in(0,1)$, then
$$
\left(\frac{x}{p}\right)^x\left(\frac{1-x}{1-p}\right)^{1-x}\ge 1.
$$
Equality holds if and only if $x= p$.
\end{lemma}

\begin{proof}
Let $p\in (0,1)$ be arbitrary, but fixed. Again, we introduce an auxiliary function,  given by
$$
g(x)\coloneqq x\ln\left(\frac{x}{p}\right)+(1-x)\ln\left(\frac{1-x}{1-p}\right),\quad x\in (0,1).
$$
Then $g$ satisfies $g(0+)= -\ln(1-p)>0$, $g(1-)=-\ln(p)>0$ and
$$
g'(x)=\ln\left(\frac{x}{p}\frac{1-p}{1-x}\right).
$$
Moreover, $g'(0+)=-\infty$, $g'(1-)=+\infty$ and $g'(x)=0$ are satisfied if and only if $x=p$. Finally, we have $g(p)=0$. Now we can easily deduce the assertions of the lemma.
\end{proof}

The preceding auxiliary results can be used to prove the following generalisation of Lemma \ref{L1}.

\begin{lemma}\label{L1p}
If $0<d\le 1\le u$ and $u\neq d$, then
$$
\lim_{n\rightarrow \infty} A(n,u,d,p)=0.
$$
\end{lemma}

\begin{proof}
The overall structure of the proof is similar to the proof of Lemma \ref{L1}.

\medskip

 \textbf{Case 1:} $pu+qd<1$. Then it follows that
$$
A(n,u,d,p)=\left(pu+qd\right)^n\cdot\sum_{k=0}^{\widetilde{k}_0(n,u,d)}\binom{n}{k}
\left(\frac{pu}{pu+qd}\right)^k\left(\frac{qd}{pu+qd}\right)^{n-k}\le \left(pu+qd\right)^n\to 0
$$
as $n\to\infty$.

\medskip

\textbf{Case 2:}  $pu+qd= 1$. We have  $d<1$, since  $u\neq d$. An application of Lemma \ref{ungl1} with  $x=qd\in (0,d)$ shows that
$$
(1-qd)^{1-qd}(pd)^{qd}>p
$$
and therefore
$$
u^{pu}d^{qd}>1.
$$
The previous inequality can be  rewritten as
$$
\frac{\ln\left(\frac{1}{d}\right)}{\ln\left(\frac{u}{d}\right)}<pu.
$$
Hence, the asymptotic behaviour of $\widetilde{k}_0$ is given by
\begin{equation}\label{b2p}
\frac{1}{n}\widetilde{k}_0(n,u,d)\to \frac{\ln\left(\frac{1}{d}\right)}{\ln\left(\frac{u}{d}\right)}<pu\quad \text{for }n\to\infty.
\end{equation}
Now let $Y_n^{(p)}\sim\Bin(n,pu)$. The law of large numbers implies that
\begin{align*}
A(n,u,d,p)&=\sum_{k=0}^{\widetilde{k}_0(n,u,d)}\binom{n}{k}
\left(pu\right)^k\left(qd\right)^{n-k}
=\BP\left(Y_n^{(p)}\le \widetilde{k}_0(n,u,d)\right)\\
&=\BP\left(\frac{1}{n}Y_n^{(p)}\le \frac{\widetilde{k}_0(n,u,d)}{n}\right)\to 0 \quad \text{as } n\to\infty,
\end{align*}
where we used \eqref{b2p} and $\frac{1}{n}\mathbb{E} Y_n^{(p)}= pu$.\medskip

Analogously to the proof of Lemma \ref{L1}, an alternative argument based on Hoeffding's or Chernoff's inequality yields an exponential rate of convergence.

\medskip

\textbf{Case 3:}  $pu+qd> 1$ and $u^pd^q\neq 1$. An application of Lemma \ref{ungl1} with $x=qd\in (0,d)$ (and hence $1-qd>0$) yields the  inequality
$$
(pu)^{1-qd}(pd)^{qd}>(1-qd)^{1-qd}(pd)^{qd}\ge p,
$$
and therefore
$$
u^{pu}d^{qd}>1.
$$
The previous inequality is equivalent to the left  inequality in
$$
\frac{\ln\left(\frac{1}{d}\right)}{\ln\left(\frac{u}{d}\right)}<\frac{pu}{pu+qd}<1.
$$
Finally, for sufficiently large $n$ we get
\begin{equation}\label{b3p}
1-\frac{\widetilde{k}_0(n,u,d)}{n}>\frac{qd}{pu+qd}\quad\text{and}\quad
\frac{\widetilde{k}_0(n,u,d)}{n}<\frac{pu}{pu+qd}.
\end{equation}
Now let $\xi_n^{(p)}\sim \Bin (n,\frac{pu}{pu+qd})$ and $\zeta_n^{(p)}\coloneqq n-\xi_n^{(p)}\sim\Bin(n,\frac{qd}{pu+qd})$.  Similarly to the proof of Theorem \ref{Thm1}, we can use
 \eqref{b3p} in the derivation of an upper bound on $A(n,u,d,p)$, that is,
\begin{align*}
A(n,u,d,p)&=\left( {pu+qd}\right)^n\BP\left(\xi_n^{(p)}\le
\widetilde{k}_0(n,u,d)\right)
=\left( {pu+qd}\right)^n\BP\left(\zeta_n^{(p)}\ge n-
\widetilde{k}_0(n,u,d)\right)\\
&=\left( {pu+qd}\right)^n\BP\left(\zeta_n^{(p)}\ge \left(1-
\frac{\widetilde{k}_0(n,u,d)}{n}\right)n\right)\\
&\le \left( {pu+qd}\right)^n
\left[\frac{\left(\frac{pu}{pu+qd}\right)^{\frac{\widetilde{k}_0(n,u,d)}{n}}
\left(\frac{qd}{pu+qd}\right)^{1-\frac{\widetilde{k}_0(n,u,d)}{n}}}
 {\left(\frac{\widetilde{k}_0(n,u,d)}{n}\right)^{\frac{\widetilde{k}_0(n,u,d)}{n}}
\left(1-\frac{\widetilde{k}_0(n,u,d)}{n}\right)^{1-\frac{\widetilde{k}_0(n,u,d)}{n}}}
\right]^n\\
&= \omega(n,u,d,p)^n,
\end{align*}
where
$$
\omega(n,u,d,p)\coloneqq  \frac{u^{\frac{\widetilde{k}_0(n,u,d)}{n}}
d^{1-\frac{\widetilde{k}_0(n,u,d)}{n}}}
{\left(\frac{1}{p}\right)^{\frac{\widetilde{k}_0(n,u,d)}{n}}\left(\frac{1}{q}\right)^{1-\frac{\widetilde{k}_0(n,u,d)}{n}}\left(\frac{\widetilde{k}_0(n,u,d)}{n}\right)^{\frac{\widetilde{k}_0(n,u,d)}{n}}
\left(1-\frac{\widetilde{k}_0(n,u,d)}{n}\right)^{1-\frac{\widetilde{k}_0(n,u,d)}{n}}}
.
$$

\medskip

Observe that
$$
\frac{1}{n}\widetilde{k}_0(n,u,d)\to \frac{\ln\left(\frac{1}{d}\right)}{\ln\left(\frac{u}{d}\right)}\eqqcolon c\in (0,1).
$$
The limit $c$ satisfies $c\neq p$ due to the assumption $u^pd^{1-p}\neq 1$. From Lemma \ref{ungl2} we deduce  that the denominator of $\omega(n,u,d,p)$ converges to $\rho(u,d,p)>1$ as $n\to\infty$.\medskip

Moreover,  \eqref{eqhilf2} remains true (see the proof of Lemma \ref{L1}), since $\widetilde{k}_0$ does not depend on $p$. \medskip

We conclude that $\omega(n,u,d,p)\to  {\rho(u,d,p)}^{-1}\in (0,1)$ and therefore
$\omega(n,u,d,p)^n\to 0$ as $n\to\infty$.

\medskip

\textbf{Case 4:} $u^pd^q= 1$. Since $d\neq u$, it follows that $pu+qd>1$. Since $u^pd^q= 1$,  $\widetilde{k}_0(n,u,d)=\lfloor np\rfloor$ and therefore $A(n,u,d,p)$ simplifies to
$$
A(n,u,d,p)=\sum_{k=0}^{\lfloor np\rfloor}
\binom{n}{k}(pu)^k(qd)^{n-k}.
$$
For $k\le \lfloor np-q\rfloor$ we can easily deduce that $a_k\coloneqq
\binom{n}{k}p^kq^{n-k}\le a_{k+1}$.  In addition,   the assumption $u^pd^q= 1$ implies that  $u^kd^{n-k}=d^{n-\frac{k}{p}}$. Hence,
\begin{align}\label{eqhilfneu2}
    A(n,u,d,p)&\le \binom{n}{\lfloor np\rfloor}
    p^{\lfloor np\rfloor}q^{n-\lfloor np\rfloor}\sum_{k=0}^{\lfloor np\rfloor}d^{n-\frac{k}{p}}
    \le
    \binom{n}{\lfloor np\rfloor}
    p^{\lfloor np\rfloor}q^{n-\lfloor np\rfloor}\frac{1}{1-d}.
\end{align}
Further, the right-hand side of \eqref{eqhilfneu2} converges to zero since
\begin{align}
  \binom{n}{\lfloor np\rfloor}
    p^{\lfloor np\rfloor}q^{n-\lfloor np\rfloor}
    &=\BP\left(\lfloor np\rfloor-1<X_n^{(p)}\le \lfloor np\rfloor
    \right)\nonumber\\
    &=\BP\left(\frac{\lfloor np\rfloor-np-1}{\sqrt{npq}}<\frac{X_n^{(p)}-np}{\sqrt{npq}}\le \frac{\lfloor np\rfloor-np}{\sqrt{npq}}
    \right)\nonumber\\
    &\to\Phi(0)-\Phi(0)=0 \quad \text{for }n\to\infty,\label{limitalter}
\end{align}
where $\Phi$ is the cumulative  distribution function of the standard normal distribution. \medskip

Combining \eqref{eqhilfneu2} and \eqref{limitalter}, we obtain that $A(n,u,d,p)$ converges to zero as $n\to \infty$. The Berry--Esseen theorem further shows that the convergence is of the order $1/\sqrt{n}$.
\end{proof}

\medskip

As explained in the remark following the proof of Lemma \ref{L1}, it is possible to include the second case of the previous proof into the third case. However, we decided to treat these cases separately due to the same reasons as mentioned before.\medskip

Instead of using the central limit theorem and the Berry--Esseen theorem at the end of the fourth case, we could instead use that
$$
 \binom{n}{\lfloor np\rfloor}
    p^{\lfloor np\rfloor}q^{n-\lfloor np\rfloor}
    \sim  \frac{1}{\sqrt{2\pi pq}}\cdot \frac{1}{\sqrt{n}}
$$
as $n\to \infty$, to prove \eqref{limitalter}. The former asymptotic equivalence arises as a local central limit theorem in the proof of the De Moivre--Laplace theorem (see \cite[p. 55]{Shir}). Alternatively, one could use Stirling's approximation (while using the binary entropy function) to provide a more direct argument.\medskip

In order to determine the asymptotic behaviour of the expected net profit, we need to identify the limit of the second summand, denoted by $B(n,u,d,p)$, as $n\to \infty$.

\begin{lemma}\label{L2p}
If $0<d\le 1\le u$ and $u\neq  d$, then
$$
\lim_{n\rightarrow \infty} B(n,u,d,p)=
\begin{cases}
0,& \text{if } \, u^pd^{q}<1,\\
1,& \text{if } \, u^pd^{q}=1,\\
2,& \text{if } \, u^pd^{q}>1.
\end{cases}
$$
\end{lemma}

\begin{proof}
\textbf{Case 1:}  $u^pd^{q}<1$. Then
\begin{equation}\label{b1p}
\frac{1}{n}\widetilde{k}_0(n,u,d)\to \frac{\ln\left(\frac{1}{d}\right)}{\ln\left(\frac{u}{d}\right)}>p\quad \text{as }n\to\infty,
\end{equation}
where the inequality giving a lower bound on the limit of $n^{-1}\widetilde{k}_0$ can  be rewritten as $u^pd^{q}<1$. Since $X_n^{(p)}\sim\Bin(n,p)$, the law of large numbers yields
$$
B(n,u,d,p)=2\,\BP\left(X_n^{(p)}\ge \widetilde{k}_0(n,u,d)+1\right)\le 2\,\BP\left(\frac{1}{n}X_n^{(p)}\ge \frac{1}{n}\widetilde{k}_0(n,u,d)\right)\to 0\quad\text{as }n\to\infty.
$$

As in the previous proofs, we  obtain an exponential rate of convergence using Chernoff's inequality in combination with  $\widetilde{k}_0(n,u,d)-np>0$  for sufficiently large $n$ and \eqref{b1p} , that is,
$$
B(n,u,d,p)\le 2\,\BP\left(X_n^{(p)}-np\ge \widetilde{k}_0(n,u,d)-np\right)
\le 4\exp\left(-2\left(\frac{\widetilde{k}_0(n,u,d)}{n}-p
\right)^2n\right)\to 0
$$
as $n\to\infty$.

\medskip

\textbf{Case 2:}  $u^pd^{q}>1$. Then it follows that
$$
\frac{1}{n}\widetilde{k}_0(n,u,d)\to \frac{\ln\left(\frac{1}{d}\right)}{\ln\left(\frac{u}{d}\right)}
<p.
$$
Similarly to the estimate in Case 1, the law of large numbers yields
$$
B(n,u,d,p)=2\, \BP(X_n^{(p)}\ge \widetilde{k}_0(n,u,d)+1)=
2\, \BP\left(\frac{1}{n}X_n^{(p)}\ge \frac{1}{n}(\widetilde{k}_0(n,u,d)+1)\right)\to 2\cdot 1=2
$$
as $n\to\infty$.

\medskip

Furthermore, by an application of  Chernoff's inequality we obtain an exponential lower bound on $B(n,u,d,p)$ given by
\begin{align*}
B(n,u,d,p) &= 2 \, \mathbb{P}\left(X_n^{(p)} \ge \widetilde{k}_0(n,u,d)+1 \right) = 2 \left(1- \mathbb{P}\left(X_n^{(p)} \le \widetilde{k}_0(n,u,d) \right) \right)\\
 &\geq 2 \left(1- 2\cdot \exp \left(-2n\left(\frac{\widetilde{k}_0(n,u,d)}{n}-p \right)^2 \right) \right) \rightarrow 2 \quad \text{as }n\to\infty.
\end{align*}
In combination with the upper bound, $B(n,u,d,p)\leq 2$, we obtain that $B(n,u,d,p)$ converges to $2$ as $n\to \infty$ at an exponential rate.

\medskip

\textbf{Case 3:}  $u^pd^{q}=1$. In this critical case, we will argue more effectively by using the central limit theorem.
The assumption $u^pd^{1-p}=1$ implies that $\widetilde{k}_0(n,u,d)=\lfloor np\rfloor$ and therefore
\begin{align*}
    B(n,u,d,p)&=2\,\sum_{k=\lfloor np\rfloor+1}^n
    \binom{n}{k}p^kq^{n-k}
    =2\,\BP\left(X_n^{(p)}\ge \lfloor np\rfloor+1\right)\\
    &=2\left(1-\BP\left(
    X_n^{(p)}\le \lfloor np\rfloor\right)\right)
    =2\left(1-\BP\left(\frac{X_n^{(p)}-np}{\sqrt{npq}}\le \frac{\lfloor np\rfloor -np}{\sqrt{npq}}    \right)
    \right)\\&
    \to 2(1-\Phi(0))=1\quad \text{as }n\to\infty.
\end{align*}
Again, we can deduce the rate of convergence, which is given by $1/\sqrt{n}$, using the Berry--Esseen theorem.
\end{proof}

Finally,  the following generalisation of Theorem \ref{Thm1} is implied by  Lemmas  \ref{L1p} and  \ref{L2p}.

\begin{theorem}\label{Thm1p}
Let $p\in (0,1)$, $0<d\le 1\le u$ and $u\neq d$. Then the expected net profit $G(n,u,d,p)= \mathbb{E} [T(n,u,d,p)]$ after $n$ rounds satisfies
$$
\lim_{n\rightarrow \infty} G(n,u,d,p)= \begin{cases}
1,&\text{if } \, u^pd^{q}>1,\\
0,&\text{if } \, u^pd^{q}=1,\\
-1,&\text{if } \, u^pd^{q}<1.
\end{cases}
$$
\end{theorem}

\subsection{Analysis of the variance}

In the final part of our analysis, we study the variance of the random variable $T_n$. Since $T_n$ was defined as the sum of random variables (and a constant, which does not affect the variance), the variance of $T_n$ is composed of the variances of the single random variables and the covariances of any two distinct  random variables. If we further use that $\{X_n^{(p)}=k\}\cap  \{X_n^{(p)}=\ell\}=\emptyset$ for $k\neq \ell$, then we obtain
\begin{align*}
 \Var(T_n)&=\sum_{k=0}^{\widetilde{k}_0(n,u,d)}\binom{n}{k}(pu^2)^k(qd^2)^{n-k}+4\,\BP(X_n^{(p)}\ge \widetilde{k}_0(n,u,d)+1)\BP(X_n^{(p)}\le \widetilde{k}_0(n,u,d))\\
&\qquad -\sum_{k=0}^{\widetilde{k}_0(n,u,d)}\binom{n}{k}^2(pu)^{2k}(qd)^{2(n-k)}\\
&\qquad -4\sum_{k=0}^{\widetilde{k}_0(n,u,d)}\binom{n}{k}(pu)^{k}(qd)^{n-k}\cdot \BP(X_n^{(p)}\ge \widetilde{k}_0(n,u,d)+1)\\
&\qquad -2\sum_{0\le k<\ell\le \widetilde{k}_0(n,u,d)}
\binom{n}{k}\binom{n}{\ell}(pu)^{k+\ell}(qd)^{2n-k-\ell},
\end{align*}
where
$$
\BP(X_n^{(p)}\ge \widetilde{k}_0(n,u,d)+1)=\sum_{k=\widetilde{k}_0(n,u,d)+1}^n\binom{n}{k}p^kq^{n-k}.
$$

Using the previous representation of $\mathbb{V}(T_n)$, we explored how the parameters $n,u,d,p$ of the game  affect the variance of $T_n$. The results can be found in the Figures  \ref{fig:konvergenz varianz elsberg} - \ref{fig varianz d variiert}. Note that we examined the entire variance as well as the behaviour of the five summands, which we denoted by $v_i(n,u,d,p)$, $i= 1,\ldots,5$. \medskip

The results of our numerical analysis motivated the following theorem concerning the asymptotic behaviour of the variance of $T_n$ as well as of the random variable $T_n$ itself (with respect to convergence in distribution). \medskip

\begin{theorem}\label{ThmVarp}
Let $p\in (0,1)$, $0<d\le 1\le u$ and $u\neq d$. Then the asymptotic behaviour of the variance of the net profit $T_n=T(n,u,d,p)$ after $n$ rounds is given by
$$
\lim_{n\rightarrow \infty} \Var(T_n)=\begin{cases}
0,&\text{if } \, u^pd^{q}\neq 1,\\
1,&\text{if } \, u^pd^{q}=1.
\end{cases}
$$
Moreover, the limiting distribution of the random variable $T_n$ is characterised by
$$
T_n\to \begin{cases}
-1,&\text{if } \, u^pd^{q}< 1,\\
Z, & \text{if } \, u^p d^q =1,\\
1,&\text{if } \, u^pd^{q}>1,
\end{cases}
$$
where the limit is to be understood in the sense of convergence in distribution.  The distribution of the random variable $Z$ is the two-point distribution given by $\mathbb{P}(Z=1) = \frac{1}{2} = \mathbb{P}(Z=-1)$.
\end{theorem}

\begin{proof}
 We define
 $$
 C_n\coloneqq C(n,u,d,p)\coloneqq\sum_{k=0}^{\widetilde{k}_0(n,u,d)}\mathbf{1}\{X_n^{(p)}=k\} u^kd^{n-k}
 $$
 and
 $$
 D_n\coloneqq D(n,u,d,p)\coloneqq2\cdot
\mathbf{1}\{X_n^{(p)}\ge \widetilde{k}_0(n,u,d)+1\}.
 $$
 Then we can write $T_n$ as
 $$
 T_n=-1+C_n+D_n,
 $$
and therefore
$$
\Var(T_n)=\Var(C_n)+\Var(D_n)+2\cdot \BCov(C_n,D_n).
$$
Using the Cauchy--Schwarz inequality, we can estimate the covariance of $C_n$ and $D_n$ using the associated variances and obtain
$$
\BCov(C_n,D_n)^2\le \Var(C_n)\cdot\Var(D_n).
$$
To determine the asymptotic behaviour of $\mathbb{V}(T_n)$, we can focus our attention on the limits of $\mathbb{V}(C_n)$ and $\mathbb{V}(D_n)$ as $n\to \infty$ (as we will see).\medskip

The variance of $C_n$ is explicitly given by
 \begin{align*}
 \Var(C_n)&=    \sum_{k=0}^{\widetilde{k}_0(n,u,d)}\binom{n}{k}p^kq^{n-k}
 u^{2k}d^{2(n-k)}
  - \sum_{k=0}^{\widetilde{k}_0(n,u,d)}\binom{n}{k}^2(pu)^{2k}(qd)^{2(n-k)}\\
 &\qquad - 2\sum_{0\le k<\ell\le \widetilde{k}_0(n,u,d)}
\binom{n}{k}\binom{n}{\ell}(pu)^{k+\ell}(qd)^{2n-k-\ell}\ge 0.
 \end{align*}
In order to show that $\Var(C_n)\to 0$, as $n\to \infty$, it suffices to prove that the first sum in the previous representation of $\Var(C_n)$ converges to zero as $n\to \infty$. Here we can apply again the result from Lemma \ref{L1}, since apparently
$\widetilde{k}_0(n,u^2,d^2) = \widetilde{k}_0(n,u,d)$  and therefore
 $$
  \sum_{k=0}^{\widetilde{k}_0(n,u,d)}\binom{n}{k}p^kq^{n-k}
 u^{2k}d^{2(n-k)}=A(n,u^2,d^2,p)\to 0\quad\text{for }n\to\infty.
 $$
The application of Lemma \ref{L1p} is permissible since $0<d^2\le 1\le u^2$ and $d^2\neq u^2$ are satisfied under the assumptions of the theorem.\medskip

Furthermore, the variance of $D_n$ is  given by
 $$
 \Var(D_n)=4\,\BP\left(X_n^{(p)}\ge \widetilde{k}_0(n,u,d)+1\right)\BP\left(X_n^{(p)}\le \widetilde{k}_0(n,u,d)\right).
 $$

To examine the asymptotic behaviour of the variance of $D_n$, we  distinguish  three cases. \medskip

 (a) If $u^pd^{q}<1$, it follows that $\BP\left(X_n^{(p)}\ge \widetilde{k}_0(n,u,d)+1\right)\to 0$ as $n\to\infty$, by the same arguments as in the first case of the proof of Lemma \ref{L2p}.\smallskip

(b) If $u^pd^{q}>1$, we get  $\BP\left(X_n^{(p)}\ge \widetilde{k}_0(n,u,d)+1\right)\to 1$  as $n\to\infty$. This assertion can be shown analogously to the second case in the proof of Lemma \ref{L2p}. Therefore, it follows that $\BP\left(X_n^{(p)}\le \widetilde{k}_0(n,u,d)\right)\to 0$  as $n\to\infty$.\smallskip

The results in (a) and (b) imply that if $u^p d^q \neq 1$, then $\Var(D_n)\to 0$ as $n\to\infty$ with an exponential rate of convergence.\smallskip

(c) Finally, we need to treat the case where $u^pd^{q}=1$. Then it follows that $\BP\left(X_n^{(p)}\ge \widetilde{k}_0(n,u,d)+1\right)\to \frac{1}{2}$ and $\BP\left(X_n^{(p)}\le \widetilde{k}_0(n,u,d)\right)\to \frac{1}{2}$ as  $n\to\infty$. These assertions can be shown similarly to the third case in the proof of Lemma \ref{L2p}. Hence, it follows that $\Var(D_n)\to 1$  as $n\to\infty$. The convergence is of order $1/\sqrt{n}$.\medskip

Now we want to find the limit in distribution of the sequence of random variables $(T_n)_{n\in \mathbb{N}}$.\medskip

If $u^p d^q \neq 1$, we can immediately deduce the limit in distribution using Theorem \ref{Thm1p} and $\Var(T_n)\to 0$ as $n\to \infty$. In this case, we conclude the formally stronger result of convergence in probability.

\medskip

Now we are left with the case $u^p d^q = 1$.\medskip

First, we recall that $\mathbb{E}[C_n]=A(n,u,d,p)\rightarrow 0$ as $n\rightarrow \infty$ (according to Lemma \ref{L1}) and $\mathbb{V}(C_n) \rightarrow 0$ as $n\rightarrow \infty$. This directly implies that $C_n$ converges in probability to zero as $n\to \infty$. Due to Sluzki's theorem \cite[S.~209]{Henze_basic} (or \cite[Chap.~5, Thm.~11.4]{Gut}, \cite[Thm.~13.18]{Klenke}) it remains to show that $D_n - 1$ converges in distribution to $Z$. For this purpose we use the characteristic functions $\varphi_{D_n-1}$ of $D_n-1$ and $\varphi_Z$ of $Z$. For $t\in \mathbb{R}$ it follows that
\begin{align*}
    \varphi_{D_n-1}(t) &= \mathbb{E}\left[\mathrm{e}^{it(D_n - 1)} \right]\\
    &= \mathbb{P}\left(X_n^{(p)} \geq \widetilde{k}_0(n,u,d) +1 \right)\mathrm{e}^{it} + \mathbb{P}\left(X_n^{(p)} \leq \widetilde{k}_0(n,u,d) \right)\mathrm{e}^{-it} \\
    & \longrightarrow \frac{1}{2}\mathrm{e}^{it} + \frac{1}{2}\mathrm{e}^{-it} = \varphi_Z(t)
\end{align*}
as $n\rightarrow \infty$, where again we used the asymptotic behaviour of $\mathbb{P}\left(X_n^{(p)} \geq \widetilde{k}_0(n,u,d) +1 \right)$ and $\mathbb{P}\left(X_n^{(p)} \leq \widetilde{k}_0(n,u,d)\right)$ for $n\rightarrow \infty$ which we derived in Part (c). Finally, the  Lévy--Cramér continuity theorem \cite[S.~21]{Henze_basic} (or \cite[Chap.~5,  Thm.~9.1]{Gut}, \cite[Thm.~15.24]{Klenke}) implies the assertion.
\end{proof}

\medskip

\noindent
\textbf{Illustrations:} The following Figures \ref{fig:konvergenz varianz elsberg} - \ref{fig varianz d variiert} display the behaviour of $\Var(T_n)$ in terms of the different underlying parameters $n,u,d$ and $p$ of the game. \medskip

Figure \ref{fig:konvergenz varianz elsberg}  shows the asymptotic behaviour of the variance of $T_n$ as $n$ tends to infinity in the initial situation described in \enquote{GREED} ($u=1.5,\, d=0.6,\, p=0.5$). The left-hand side illustrates the entire variance while the right-hand shows the behaviour of the five summands introduced at the beginning of this subsection. As we would expect according to Theorem \ref{ThmVarp}, the convergence of the variance towards zero is clearly visible.\medskip

The remaining Figures \ref{fig varianz p variiert} - \ref{fig varianz d variiert} display the variance of $T_n$ for $n=200$ in terms of the parameters $p$, $u$ and $d$. Again, the left-hand side shows the entire variance while the right-hand side illustrates the five summands $v_1,\ldots,v_5$ separately. \medskip

The right-hand side of Figures \ref{fig varianz p variiert} - \ref{fig varianz d variiert} suggest that the quantity $v_2$ has the strongest influence on the variance. In contrast, the terms $v_1$, $v_3$ and $v_5$ only take values close to zero. These observations coincide with the proof of Theorem \ref{ThmVarp}, where we proved that $v_1,\, v_3,\, v_4$ and $v_5$ converge towards zero for any admissible choice of parameters while $v_2$ converges towards $1$ in the special case $u^p d^q = 1$. \medskip

Moreover, we want to point out that the threshold phenomenon described in Theorem  \ref{ThmVarp} is already clearly visible after $n=200$ rounds.

\begin{figure}[!bth]
  \centering
   \subfigure{\includegraphics[width = 0.45\textwidth]{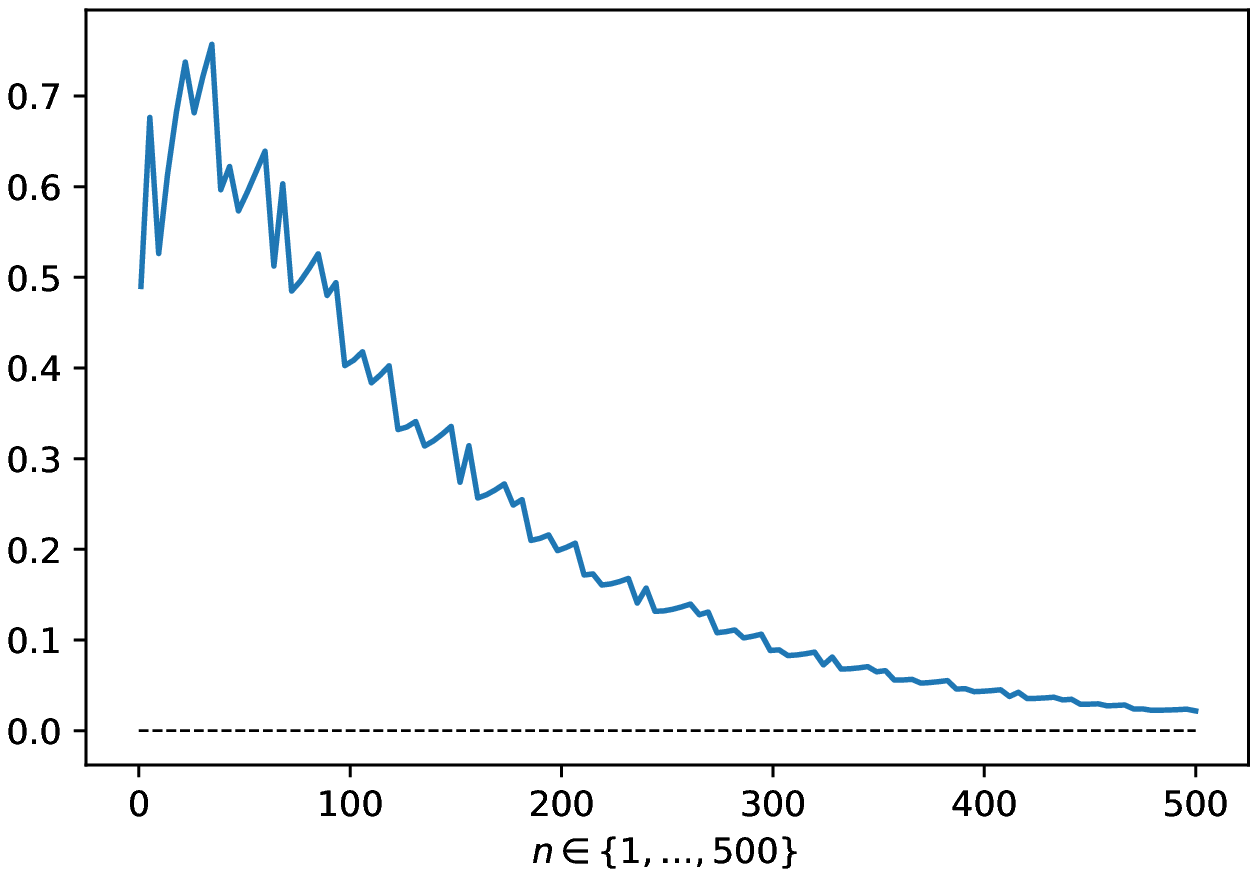}}\qquad
   \subfigure{\includegraphics[width = 0.45\textwidth]{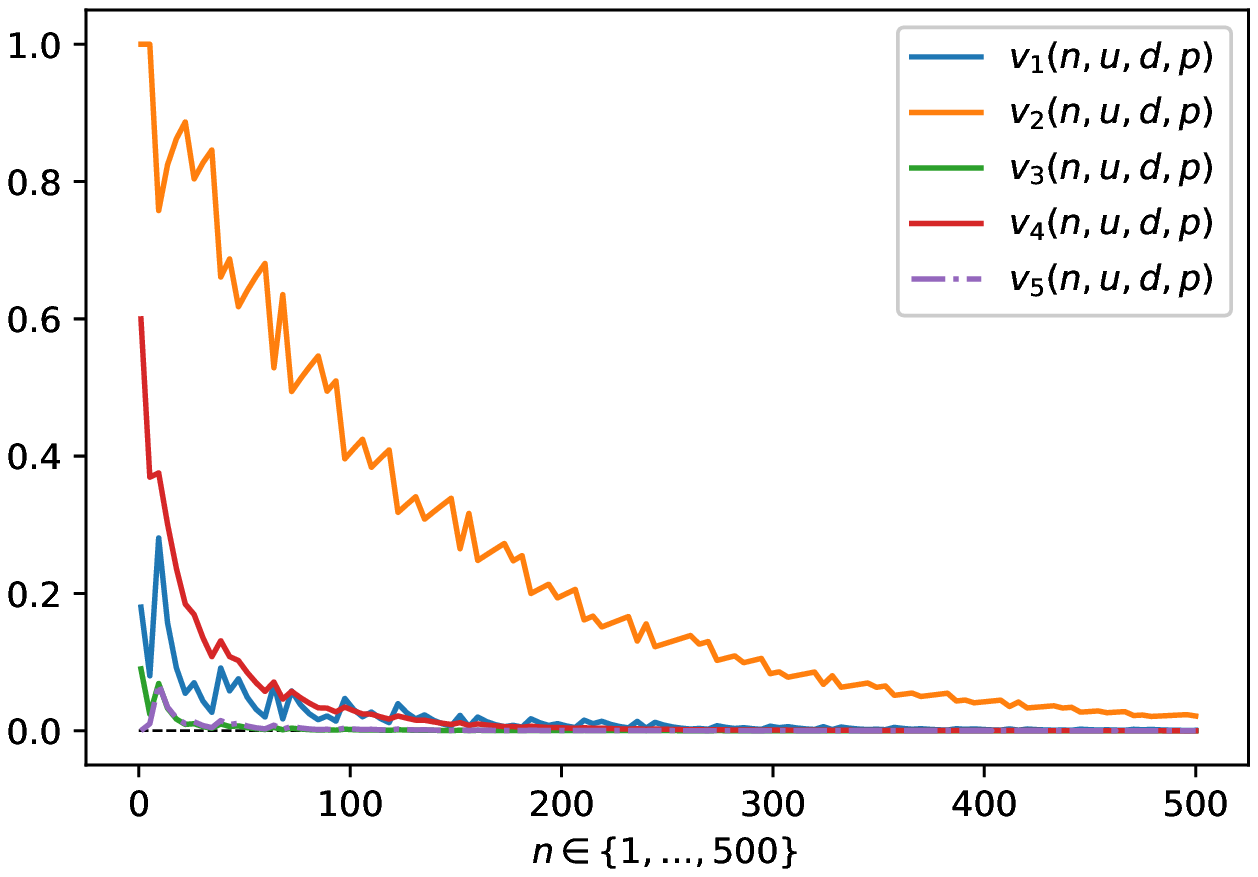}}
  \caption{Asymptotic behaviour of the variance of $T_n$  for $u=1.5$, $ d=0.6$,  $p=0.5$ (as in \cite{Els}) as $n\rightarrow \infty$ (for clarity, only every fifth value was shown). On the left is the variance itself, on the right the individual summands $v_1,\ldots,v_5$ that contribute to the variance.}
  \label{fig:konvergenz varianz elsberg}
\end{figure}

\begin{figure}[!bth]
  \centering
   \subfigure{\includegraphics[width = 0.45\textwidth]{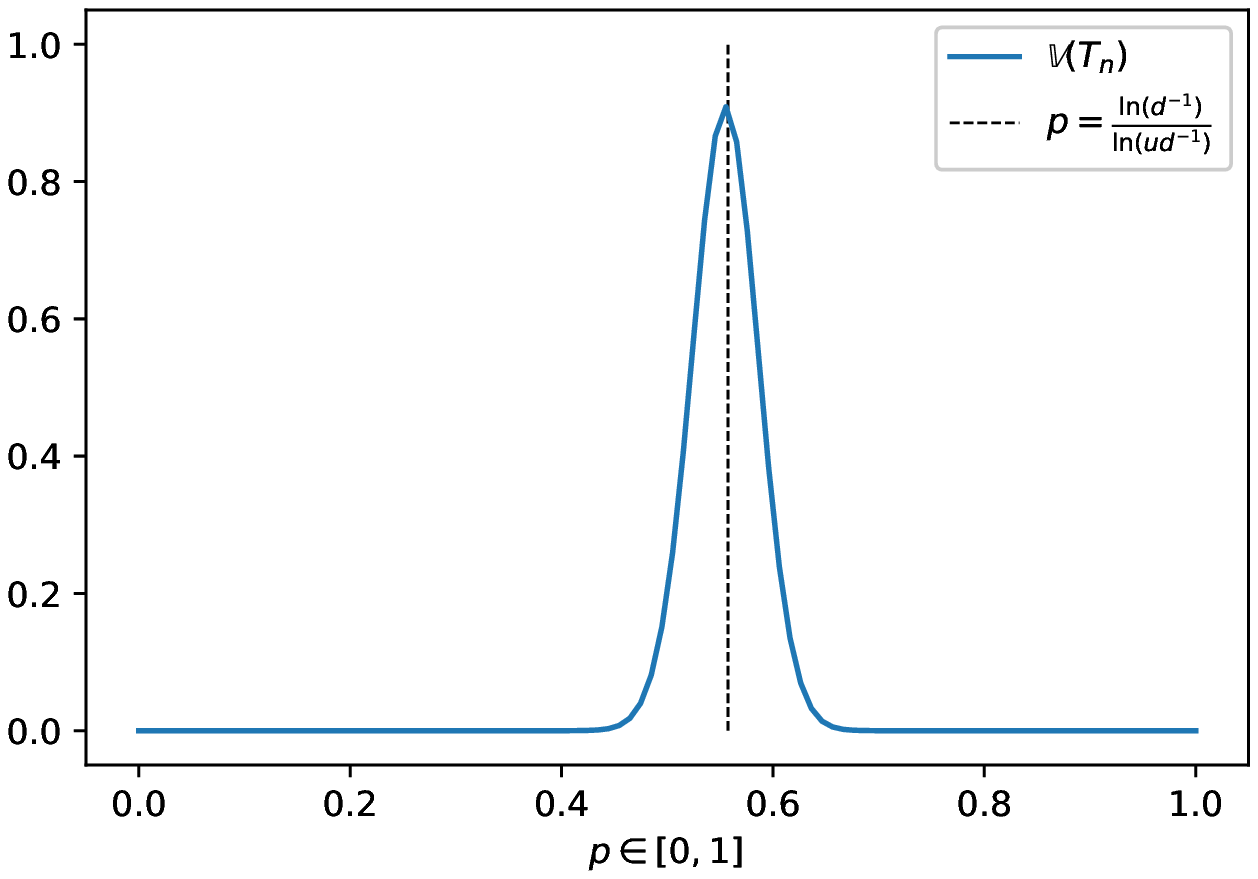}}\qquad
   \subfigure{\includegraphics[width = 0.45\textwidth]{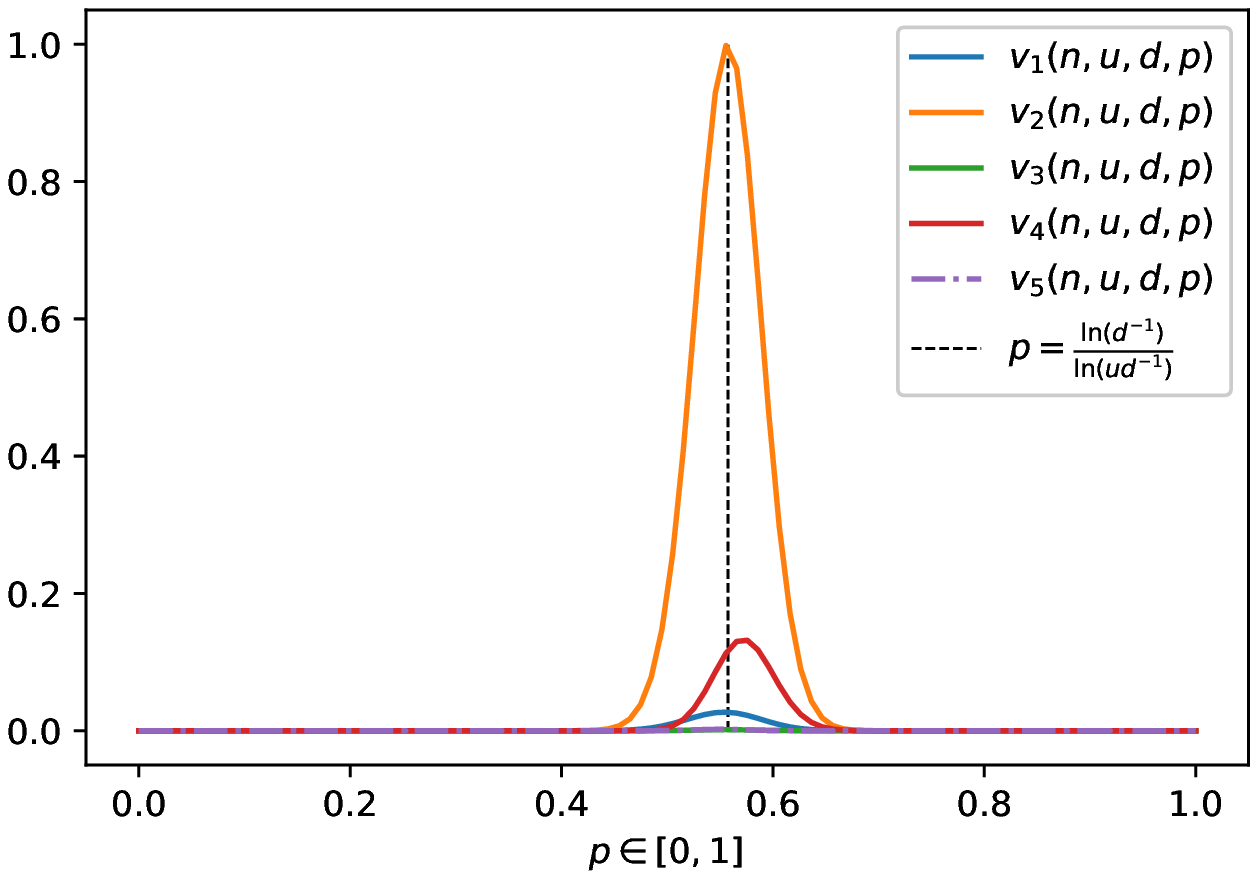}}
  \caption{Illustration of the variance of $T_n$ as a function in $p$, where $u=1.5,\, d=0.6,\, n=200$. The left-hand side shows the entire variance while the right-hand side displays the summands $v_1,\ldots,v_5$ separately.}
  \label{fig varianz p variiert}
\end{figure}

\begin{figure}[!bth]
  \centering
   \subfigure{\includegraphics[width = 0.45\textwidth]{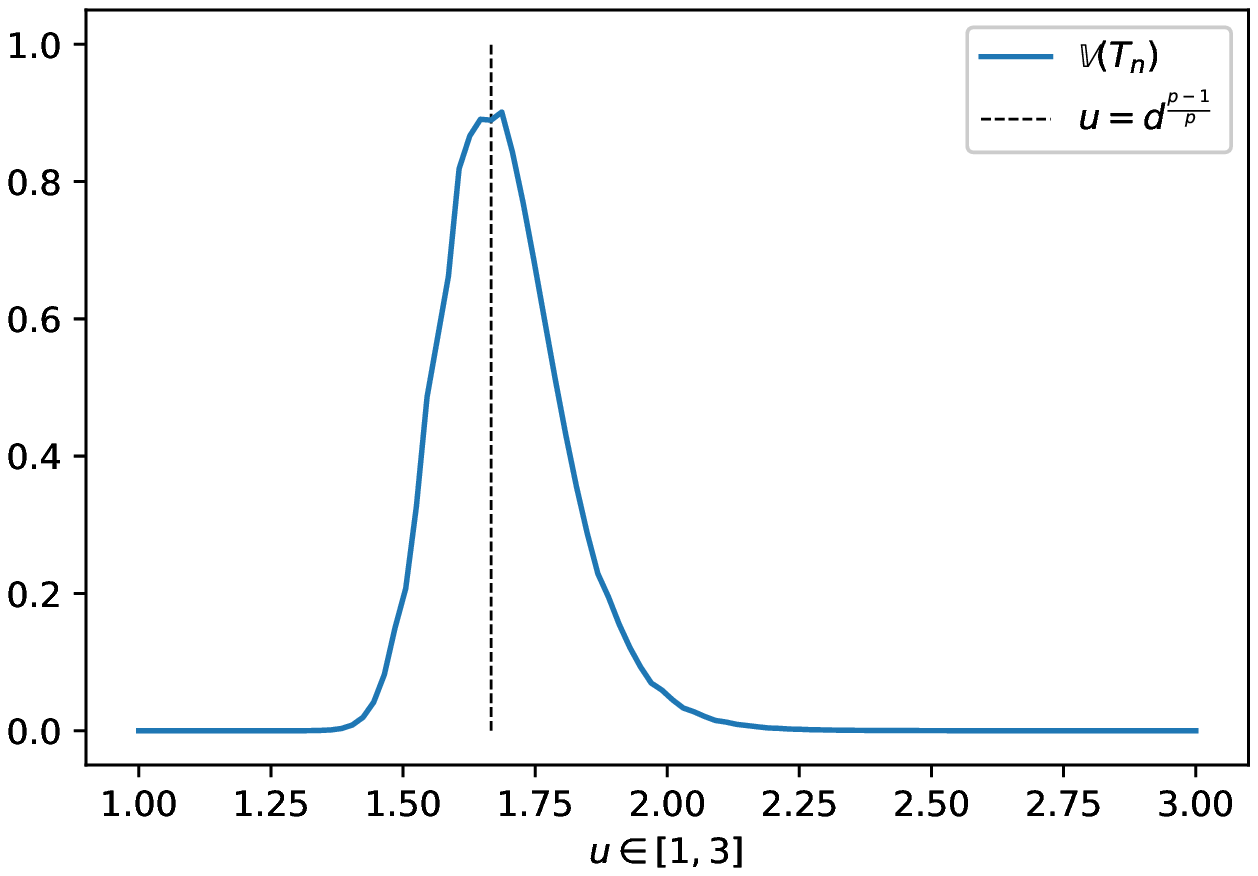}}\qquad
   \subfigure{\includegraphics[width = 0.45\textwidth]{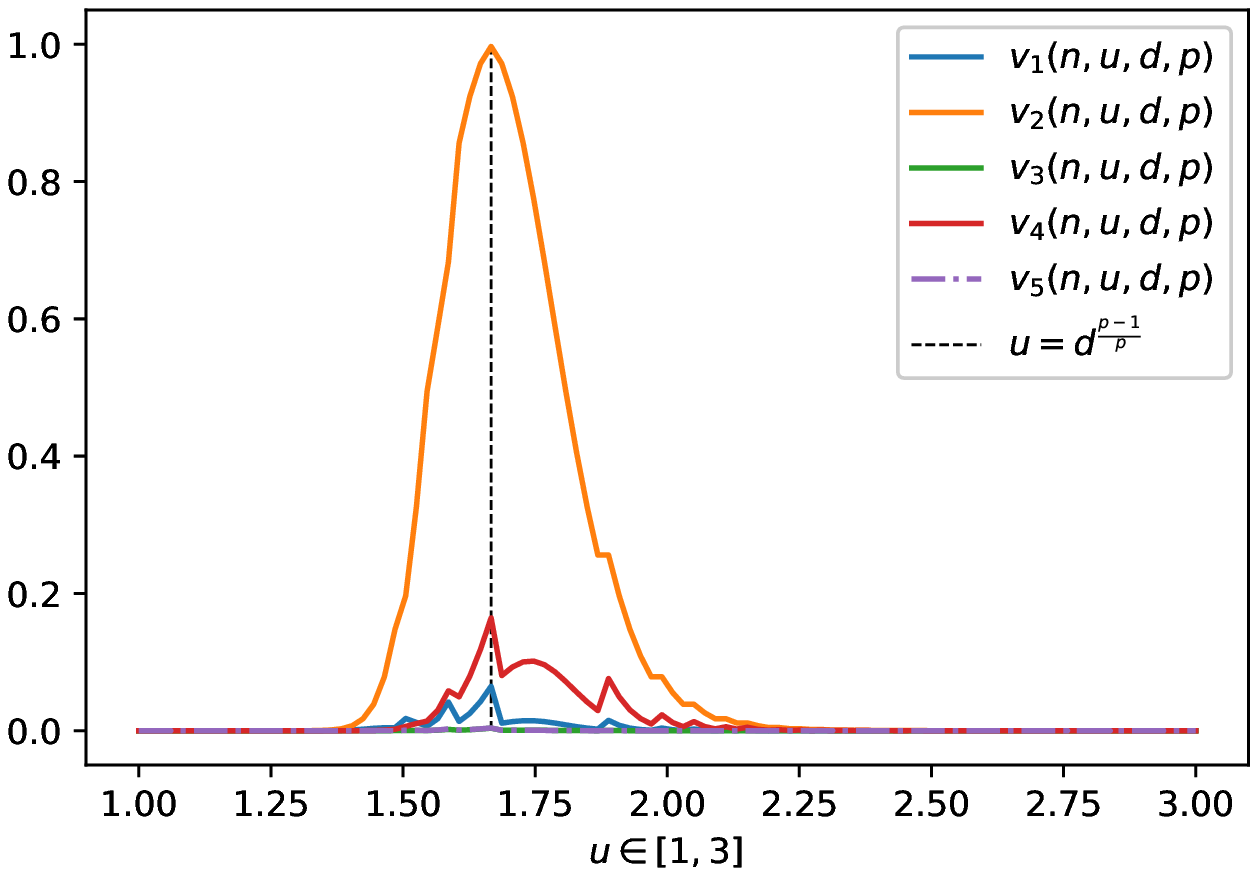}}
  \caption{Illustration of the variance of $T_n$ as a function in $u$, where $d=0.6,\, n=200,\, p=0.5$. The left-hand side shows the entire variance while the right-hand side displays the summands $v_1,\ldots,v_5$ separately.}
  \label{fig varianz u variiert}
\end{figure}

\begin{figure}[!bth]
  \centering
   \subfigure{\includegraphics[width = 0.45\textwidth]{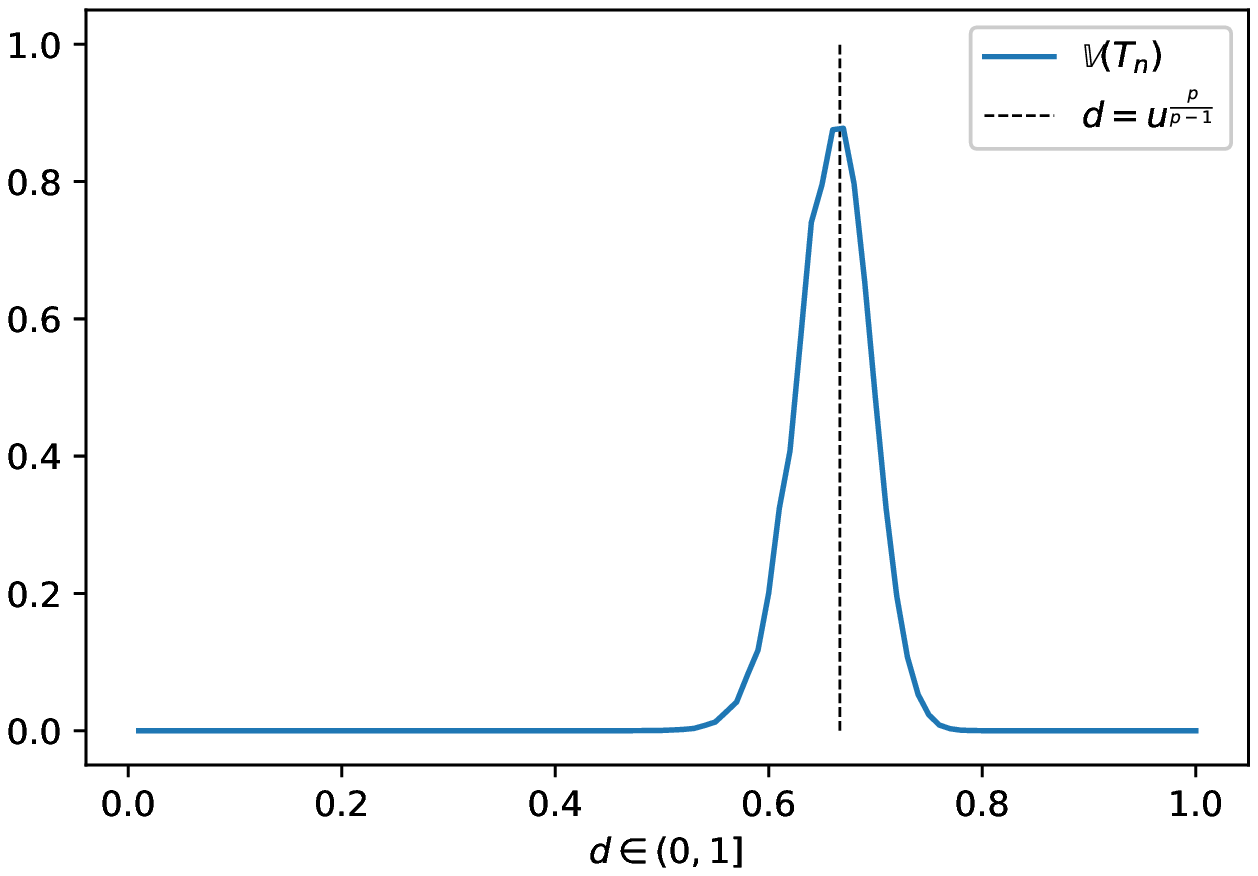}}\qquad
   \subfigure{\includegraphics[width = 0.45\textwidth]{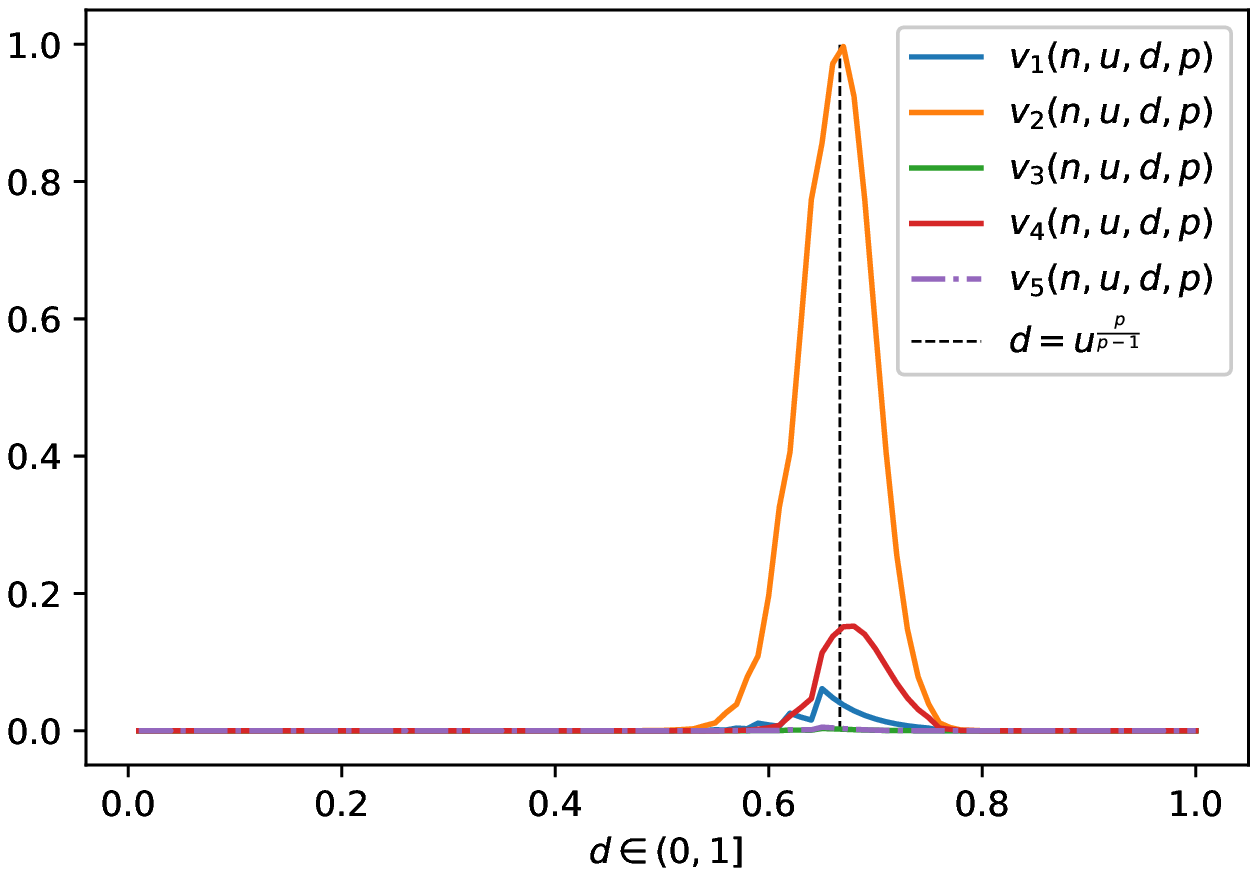}}
  \caption{Illustration of the variance of $T_n$ as a function in $d$, where $u=1.5,\, n=200,\, p=0.5$. The left-hand side shows the entire variance while the right-hand side displays the summands $v_1,\ldots,v_5$ separately.}
  \label{fig varianz d variiert}
\end{figure}

\section{Some simulations and a generalised birthday problem}\label{sec_simulationen}

The following section includes some simulations of the  development of the score over time as well as the simulation of the final net profit in $100$ repetitions of Elsberg's gamble. As we are working in the initial situation described by Elsberg, the underlying scenario is characterised by the game parameters $a=100$, $n=100$, $u=1.5$, $d=0.6$ and $p=0.5$.\medskip

The simulations were generated using Python.\medskip

The \texttt{random.binomial(n,p)} function integrated in the Numpy library was used to generate binomially distributed pseudo-random numbers.

\medskip

Figure \ref{fig simulationen zeitlicher verlauf}  shows eight simulations of the score over time, that is, as a function of  successive rounds. It should not come as a surprise that at least two of the eight simulations exhibit the same final score. The final score only depends on the number $k$ of the $n$ rounds in which the coin shows \enquote{heads}. The probability of the event of  \enquote{heads} occurring $ k $ times for a fair coin is just $ p_k =  \binom{n}{k} 2^{-n} $ for $ k \in \{0,1, \ldots, n \} $. If the eight simulations are done independently, the probability $ P_8 $ that all eight final scores are different is $P_8=8!\sum_{|I|=8}\prod_{i\in I}p_i$, where the summation extends over all $ 8 $-element subsets $ I $ of $  \{0,1,\ldots,n\}$. Maclaurin's inequality \cite[(5)]{Rosset}
 implies that
$$
\left[\frac{1}{\binom{n+1}{8}}\sum_{|I|=8}\prod_{i\in I}p_i\right]^{\frac{1}{8}}\le \frac{1}{n+1}\sum_{i=0}^np_i=\frac{1}{n+1},
$$
hence $P_8\le 8!\binom{n+1}{8}\left(\frac{1}{n+1}\right)^8$ (see  \cite{Holst,Joag} for alternative arguments), where equality holds for positive probabilities $p_i$ if and only if $p_0=\cdots=p_n={1}/{(n+1)}$. For $n=100$,
$$1-P_8\ge 1-\frac{100\cdots 94}{101^7}\ge 0.24$$
thus is a lower bound for the probability that for at least two of eight independent simulations the same final score is attained. In other words, we have estimated the probability of at least two equal final scores for non-uniform random variables in terms of the uniform case. The present question represents a \textit{generalised birthday problem} (see \cite {Henze, Mandjes, Stein}). A recursive numerical calculation with the help of \cite[Proposition 3.1]{Mase} results in $ 1-P_8 \approx 0.83 $ and is therefore considerably larger than in the case of uniform distributions. The naive direct calculation of this probability by summation over all $ 8 $-element subsets of $ \{0,1,\ldots, 100 \} $ turns out to be infeasible, so recursive methods and approximations \cite{Mase, Nunnikhoven, Stein} become relevant.

\medskip

Subsequently, the scores at the end of each round were calculated as part of a simulation study over $ 100$ rounds.

\medskip

In Figure \ref{fig simulation win} the net profits are shown which are  achieved in $ 100 $ repetitions of Elsberg's game of chance.  In exactly $ 14 $ of the $ 100 $ simulations of the gamble, the maximal net profit of  $ 100$ euros is realized.

\begin{figure}[!th]
  \centering
   \subfigure{\includegraphics[width = 0.45\textwidth]{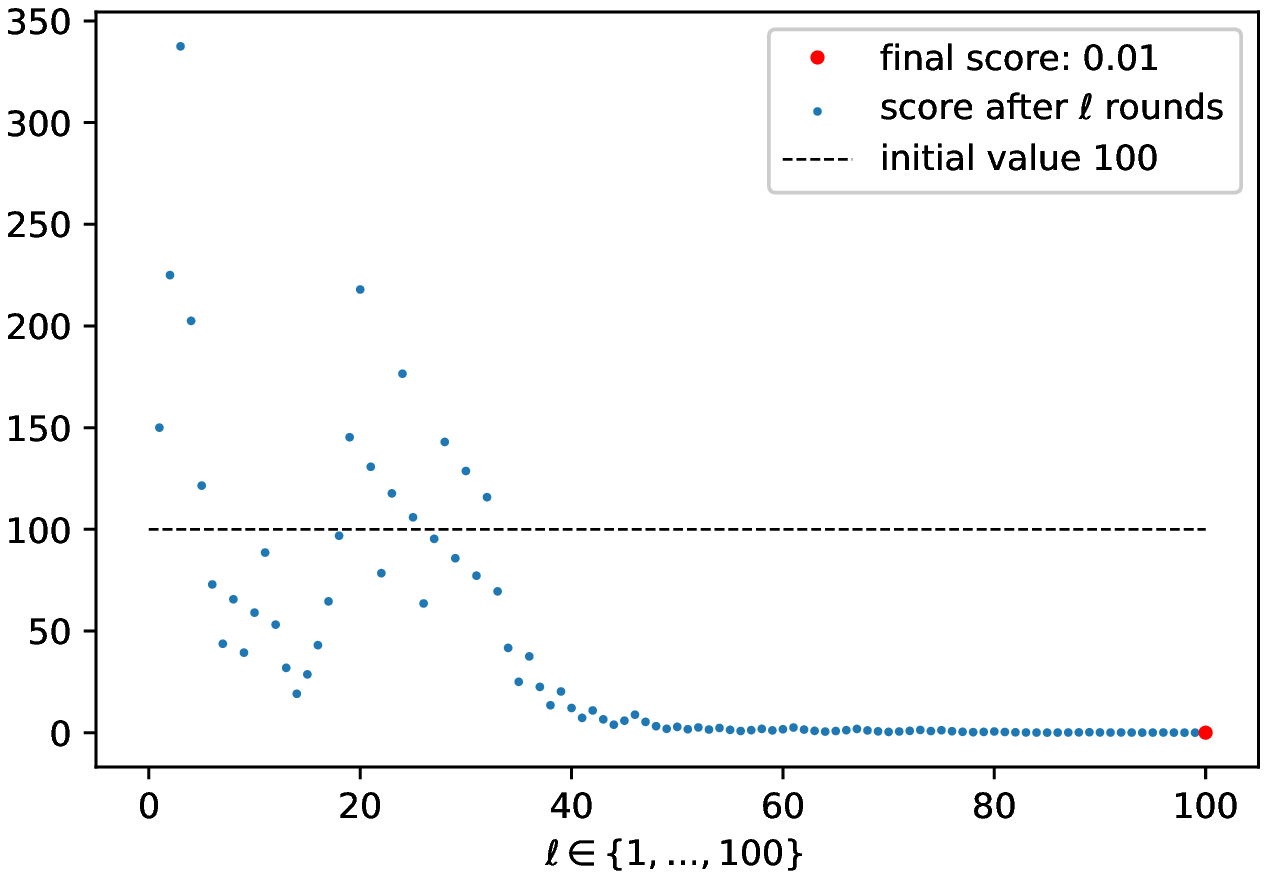}}\qquad
   \subfigure{\includegraphics[width = 0.45\textwidth]{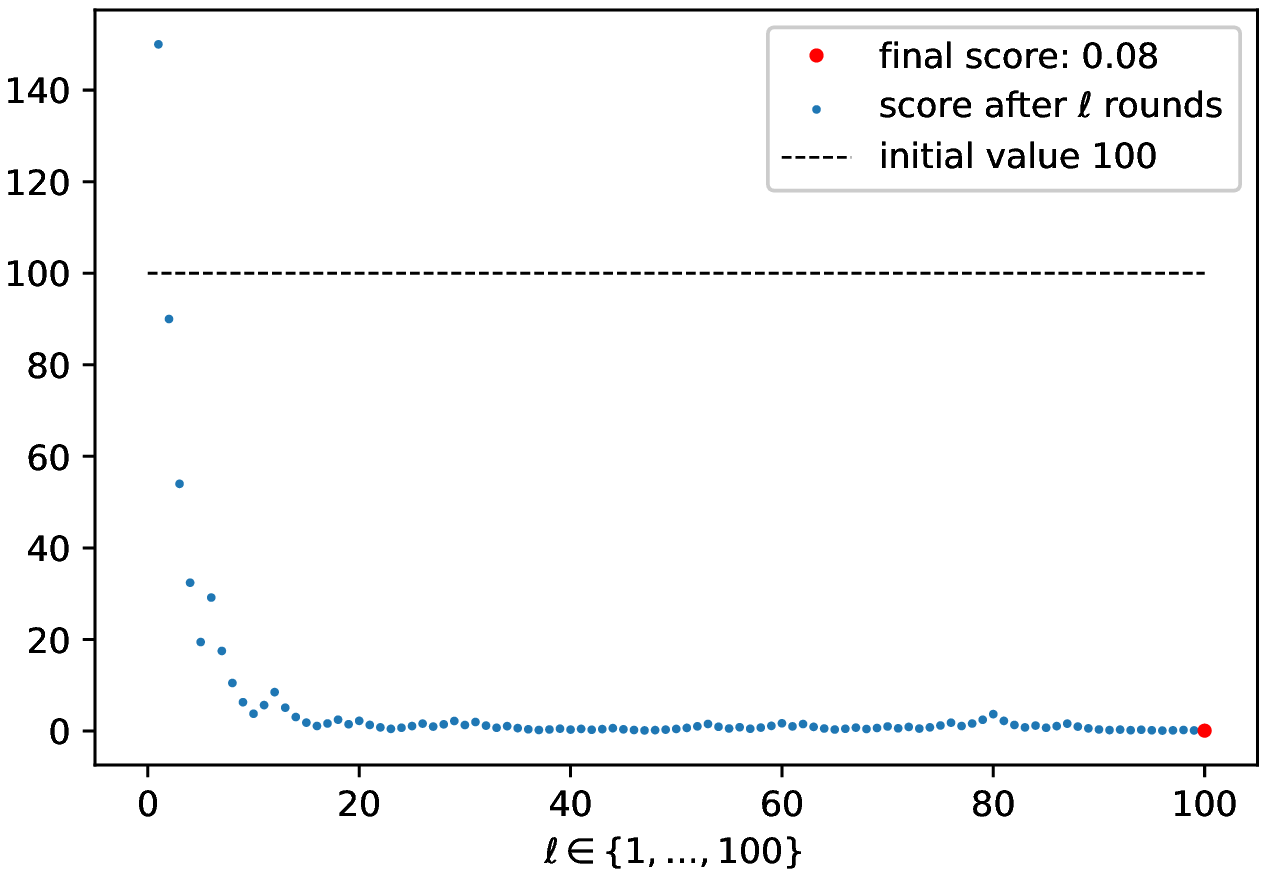}}\\
   \subfigure{\includegraphics[width = 0.45\textwidth]{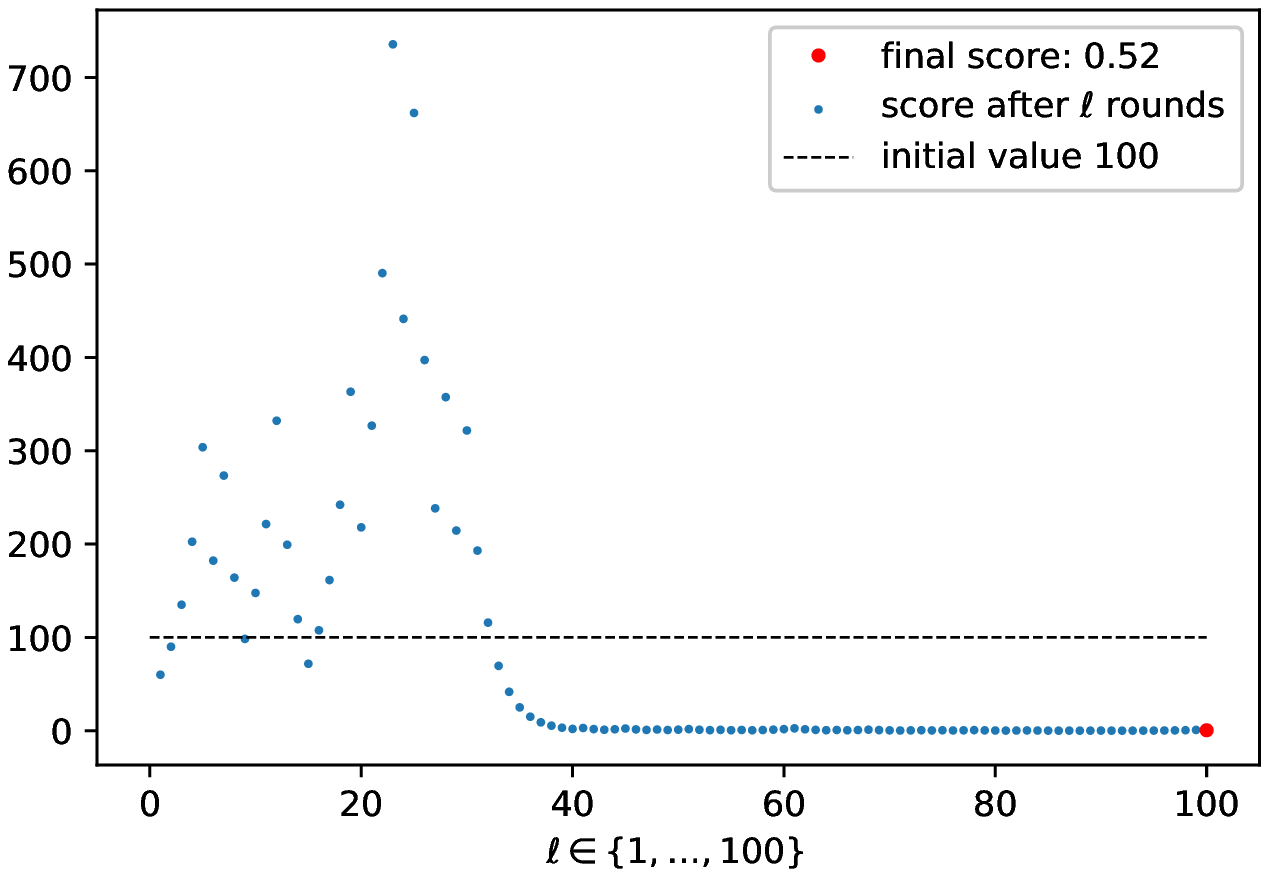}}\qquad
   \subfigure{\includegraphics[width = 0.45\textwidth]{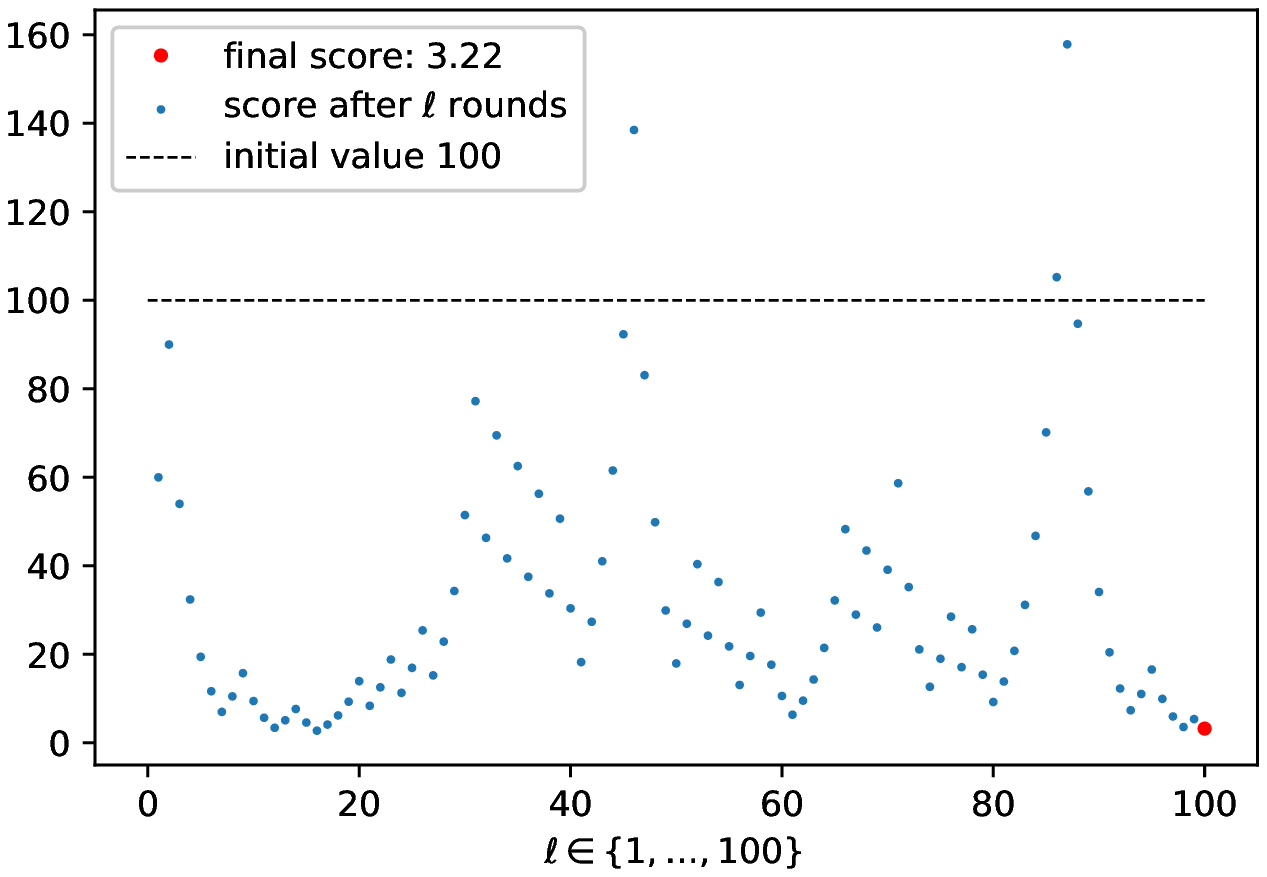}}\\
   \subfigure{\includegraphics[width = 0.45\textwidth]{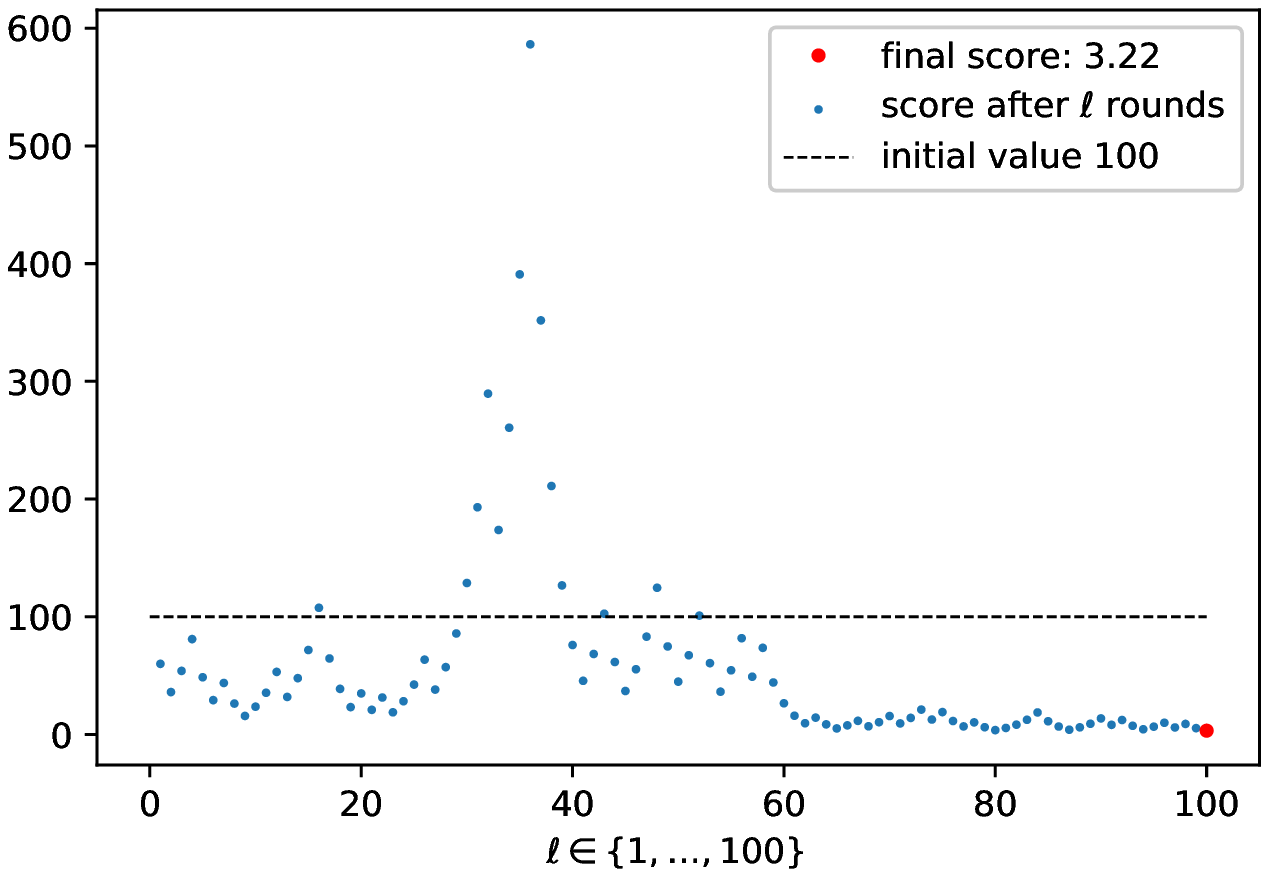}}\qquad
   \subfigure{\includegraphics[width = 0.45\textwidth]{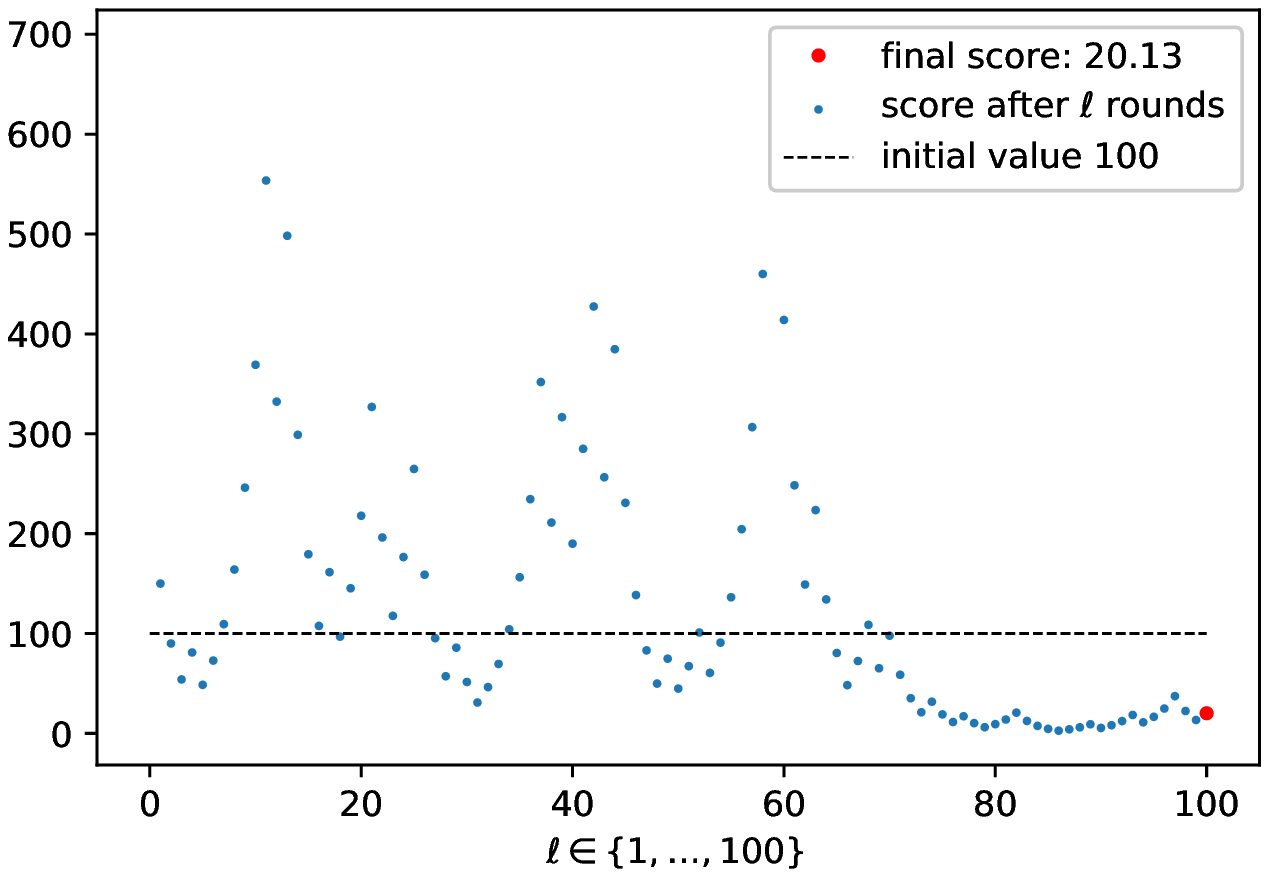}}\\
   \subfigure{\includegraphics[width = 0.45\textwidth]{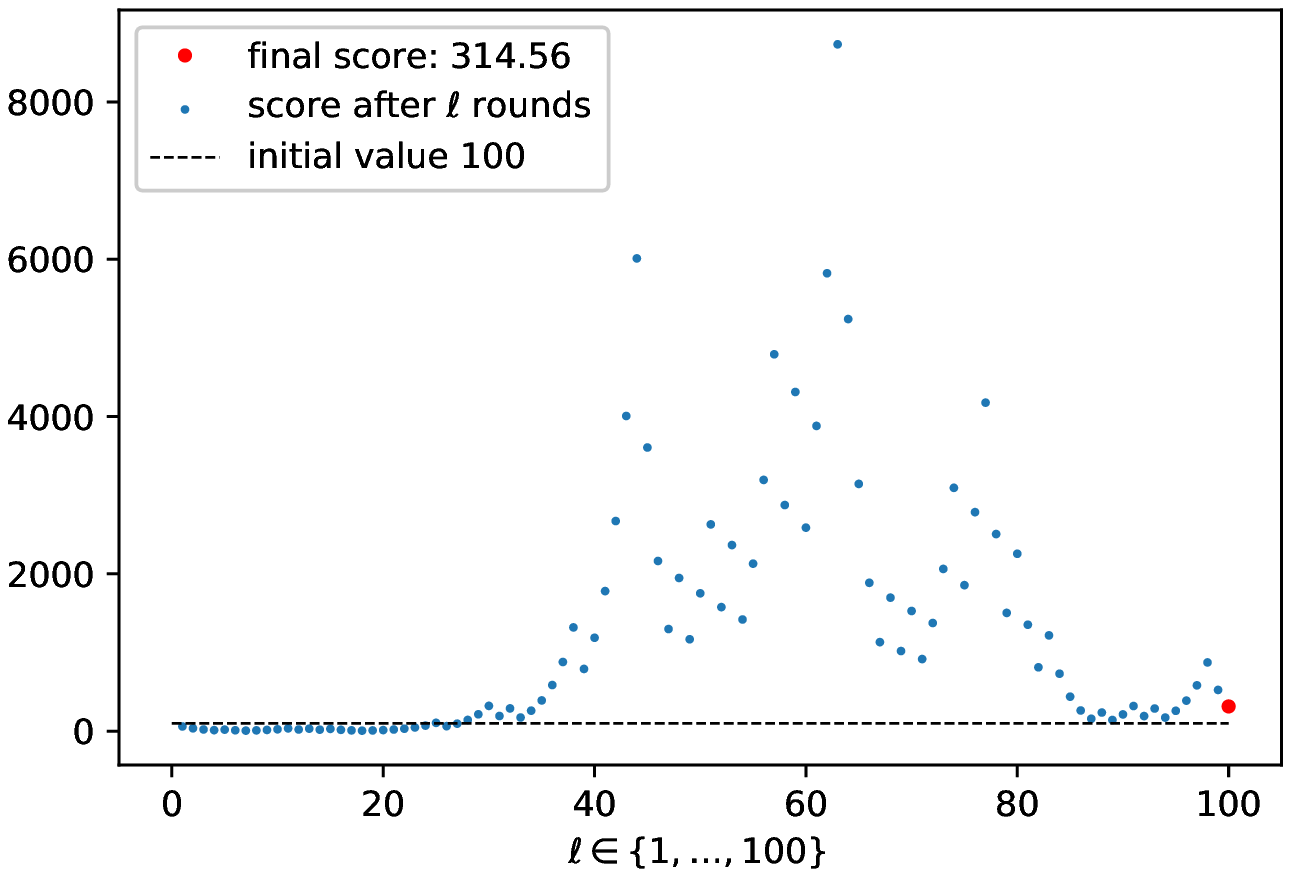}}\qquad
   \subfigure{\includegraphics[width = 0.45\textwidth]{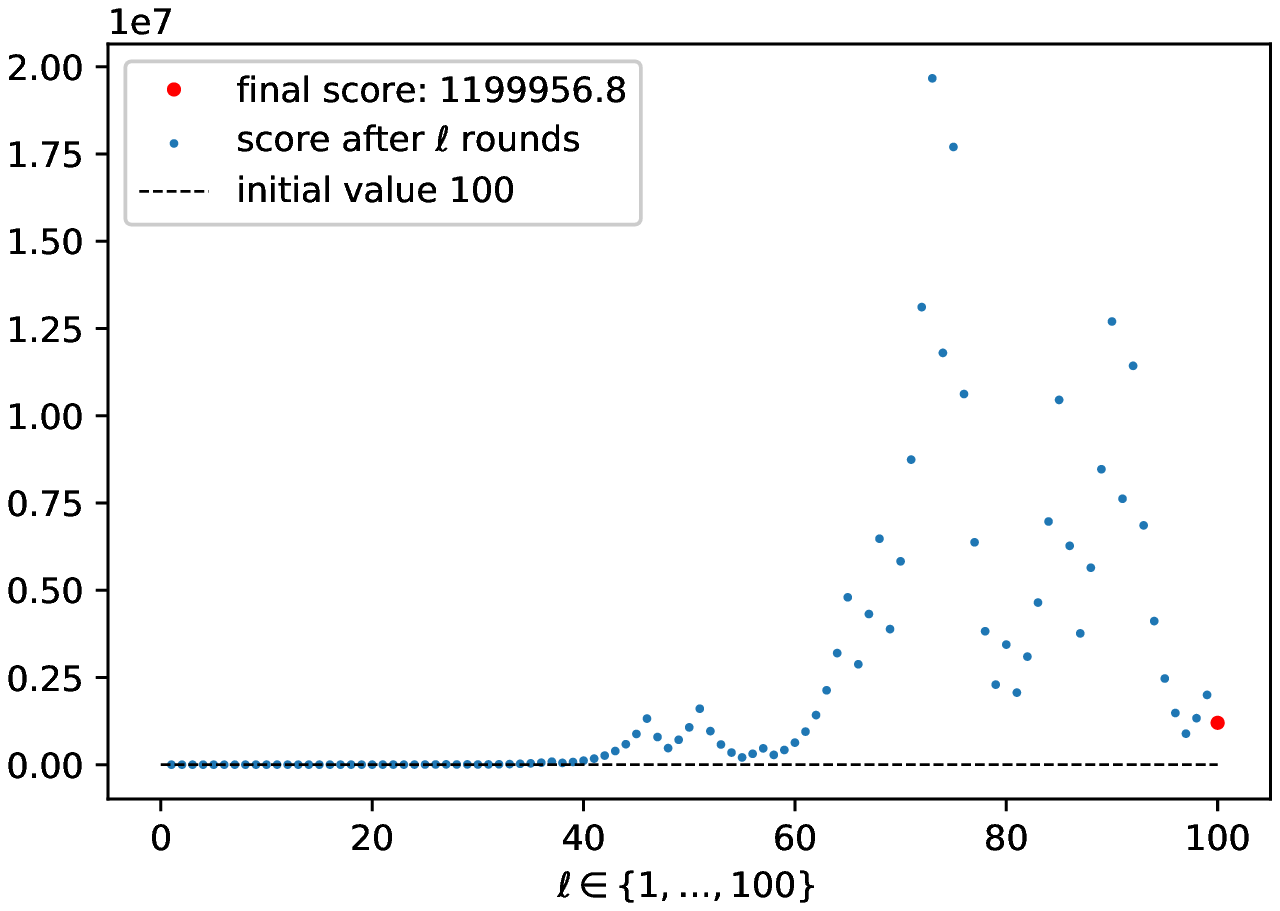}}
  \caption{Eight simulations of the score over time as a function of   successive rounds for $u=1.5$, $d=0.6$,  $p=0.5$, $a=100$ and $n=100$ (as in \cite{Els}).}
  \label{fig simulationen zeitlicher verlauf}
\end{figure}

\begin{figure}[!tbh]
	\centering
	\includegraphics[width = 0.9\textwidth]{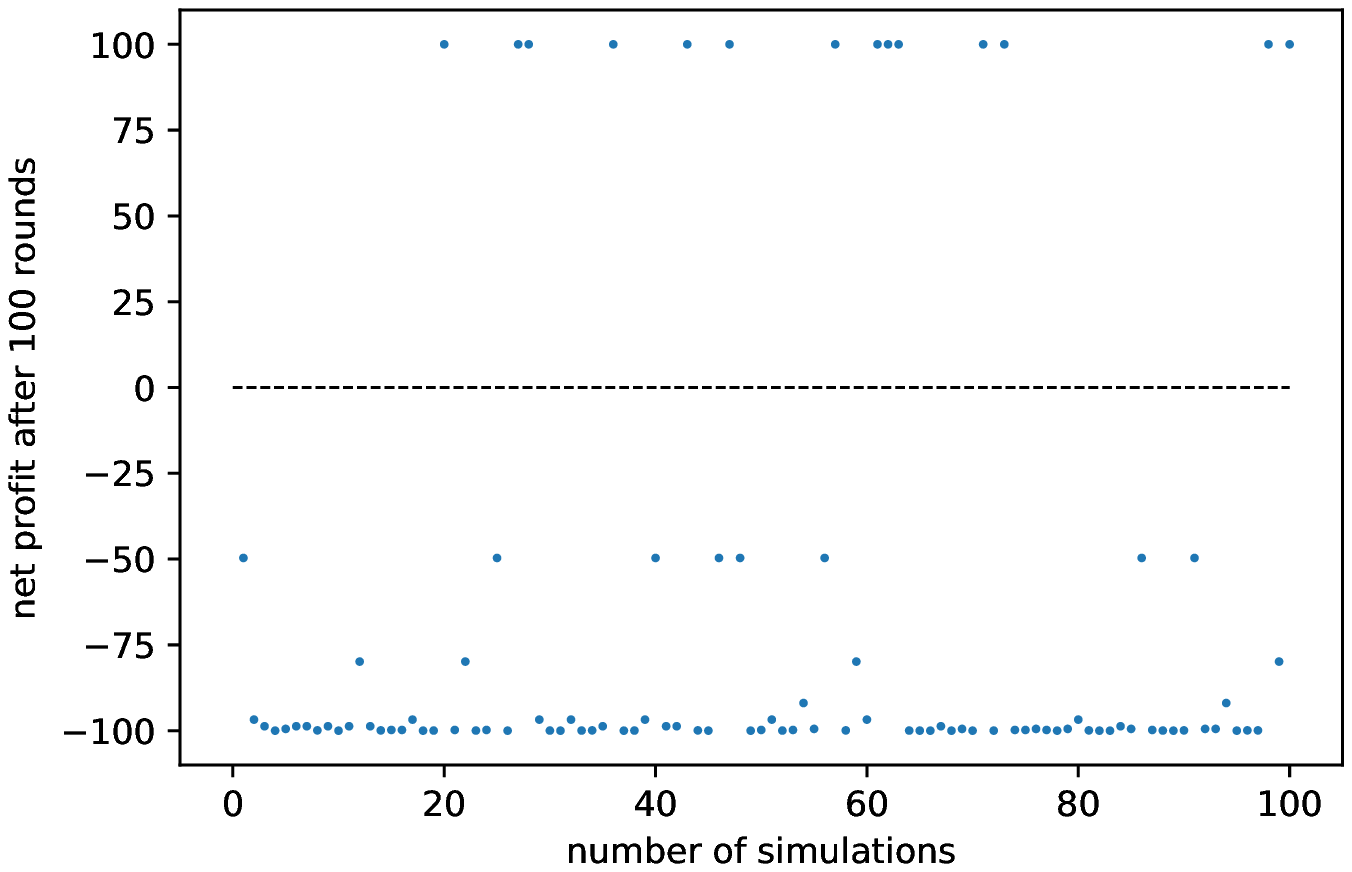}
	\caption{Net profit achieved in $100$ simulations of Elsberg's game \cite{Els} with $a=100$ and $n=100$.}
	\label{fig simulation win}
\end{figure}

\bigskip

\bigskip

\noindent
\textbf{Acknowledgement.}
The authors would like to thank Annette Hug for the idea for this project and numerous inspiring discussions. The authors are also grateful to Richard Gardner, Norbert Henze and Rolf Schneider for their friendly feedback and support.

\medskip

\textbf{Funding.}  Daniel Hug was supported by research grant HU 1874/5-1 (DFG).

\medskip

\textbf{Conflicts of interests.} The authors declare that they have no conflicts of interests and no conflicts of competing interests.

\end{document}